\documentclass[11pt]{amsart} 
\usepackage{color,graphicx,amssymb,amsmath,amsfonts,amsthm,float}
\usepackage{hyperref}
\hypersetup{colorlinks}
\usepackage[titletoc,title]{appendix}
\title[Spectral thresholds in the bipartite stochastic block model]{Spectral thresholds in the bipartite stochastic block model}
\usepackage{times,comment}
\newcommand{\del}{\delta}

\newcommand{\E}{\mathbb E}

\newcommand{\var}{\text{Var}}
\newcommand{\eps}{\epsilon}
\newcommand{\lam}{\lambda}

\newcommand{\pois}{\text{Poisson}}
\newtheorem{question}{Question}
\newtheorem{thm}{Theorem}
\newtheorem{conj}{Conjecture}

\newtheorem{lemma}{Lemma}

 \author{Laura Florescu} 
 \email{florescu@cims.nyu.edu}
 \address{New York University}
 
\author{Will Perkins} 
\email{william.perkins@gmail.com}
\address{University of Birmingham}

\begin{document}

\maketitle

\begin{abstract} 
We consider a bipartite stochastic block model on vertex sets $V_1$ and $V_2$, with planted partitions in each, and ask at what densities efficient algorithms can recover the partition of the smaller vertex set.

When $|V_2| \gg |V_1|$, multiple thresholds emerge.  We first locate a sharp threshold for detection of the partition, in the sense of the results of \cite{mossel2012stochastic,mossel2013proof} and  \cite{massoulie2014community} for the stochastic block model.  We then show that at a higher edge density, the singular vectors of the rectangular biadjacency matrix exhibit a localization / delocalization phase transition, giving recovery above the threshold and no recovery below.  Nevertheless, we propose a simple spectral algorithm,  Diagonal Deletion SVD, which recovers the partition at a nearly optimal edge density.   

The bipartite stochastic block model studied here was used by \cite{feldman2014algorithm} to give a unified algorithm for recovering planted partitions and assignments in random hypergraphs and random $k$-SAT formulae respectively.  Our results give the best known bounds for the clause density at which solutions can be found efficiently in these models as well as showing a barrier to further improvement via this reduction to the bipartite block model.
\end{abstract}
\footnote{Accepted for presentation at Conference on Learning Theory (COLT) 2016}


\section{Introduction}


The stochastic block model  is a widely studied model of community detection in random graphs, introduced by \cite{hll}. A simple description of the model is as follows: we start with $n$ vertices, divided into two or more communities, then add edges independently at random, with probabilities depending on which communities the endpoints belong to. The algorithmic task is then to infer the communities from the graph structure. 

A different class of models of random computational problems with planted solutions is that of planted satisfiability problems: we start with an assignment $\sigma$ to $n$ boolean variables and then choose clauses independently at random that are satisfied by $\sigma$. The task is to recover $\sigma$ given the random formula.  A closely related problem is that of recovering the planted assignment in \cite{goldreich2000candidate}'s one-way function, see Section~\ref{sec:goldreich}. 

A priori, the stochastic block model and planted satisfiability may seem only tangentially related. Nevertheless, two observations reveal a strong connection: 
\begin{enumerate}
\item Planted satisfiability can be viewed as a $k$-uniform hypergraph stochastic block model, with the set of $2n$ booleans literals partitioned into two communities of true and false literals under the planted assignment, and clauses represented as hyperedges.
\item \cite{feldman2014algorithm} gave a general algorithm for a unified model of planted satisfiability problems which reduces a random formula with a planted assignment to a bipartite stochastic block model with planted partitions in each of the two parts.  
\end{enumerate}

The bipartite stochastic block model in \cite{feldman2014algorithm} has the distinctive feature that the two sides of the bipartition are extremely unbalanced; in reducing from a planted $k$-satisfiability problem on $n$ variables, one side is of size $\Theta(n)$ while the other can be as large as $\Theta(n^{k-1})$.  

We study this bipartite block model in detail, first locating a sharp threshold for detection and then studying the performance of spectral algorithms.   

Our main contributions are the following:
\begin{enumerate}
\item When the ratio of the sizes of the two parts diverge, we locate a sharp threshold below which detection is impossible and above which an efficient algorithm succeeds (Theorems~\ref{thm:detection} and \ref{thm:noDetect}).  The proof of impossibility follows that of \cite{mossel2012stochastic} in the stochastic block model, with the change that we couple the graph to a broadcast model on a two-type Poisson Galton-Watson tree.  The algorithm we propose involves a reduction to the stochastic block model and the algorithms of \cite{massoulie2014community,mossel2013proof}. 
\item We next consider spectral algorithms and show that computing the singular value decomposition (SVD) of the biadjacency matrix $M$ of the model can succeed in recovering the planted partition even when the norm of the `signal', $\| \E M \|$, is much smaller than the norm of the `noise', $\| M - \E M \|$ (Theorem~\ref{thm:main}). 
\item We show that at a sparser density, the SVD fails due to a localization phenomenon in the singular vectors: almost all of the weight of the top singular vectors is concentrated on a vanishing fraction of coordinates (Theorem~\ref{thm:SVDfail}).
\item We propose a modification of the SVD algorithm, Diagonal Deletion SVD, that succeeds at a sparser density still, far below the failure of the SVD (Theorem~\ref{thm:main}).
\item We apply the first algorithm to planted hypergraph partition and planted satisfiability problems to find the best known general bounds on the density at which the planted partition or assignment can be recovered efficiently (Theorem~\ref{thm:ksat}).  
\end{enumerate}

\section{The model and main results}

\paragraph{The bipartite stochastic block model}
Fix parameters $\del \in [0,2]$, $n_1\le n_2$, and $p \in [0,1/2]$. Then we define the bipartite stochastic block model as follows: 
\begin{itemize}
\item Take two vertex sets $V_1, V_2$, with $|V_1| = n_1$, $|V_2|=n_2$.
\item Assign labels `+' and `-' independently with probability $1/2$ to each vertex in $V_1$ and $V_2$.  Let $\sigma \in \{ \pm 1 \}^{n_1}$ denote the labels of the vertices in $V_1$ and $\tau \in \{ \pm 1 \}^{n_2}$ denote the labels of $V_2$. 
\item Add edges independently at random between $V_1$ and $V_2$ as follows: 
for $u \in V_1, v \in V_2$ with $\sigma(u) = \tau (v)$, add the edge $(u,v)$ with probability $\del p$; for $\sigma(u) \ne \tau(v)$, add $(u,v)$ with probability $(2-\del)p$.  
\end{itemize}
\noindent
\textbf{Algorithmic task:} Determine the labels of the vertices given the bipartite graph, and do so with an efficient algorithm at the smallest possible edge density $p$.

\begin{figure}
\centering
\includegraphics[scale=0.3]{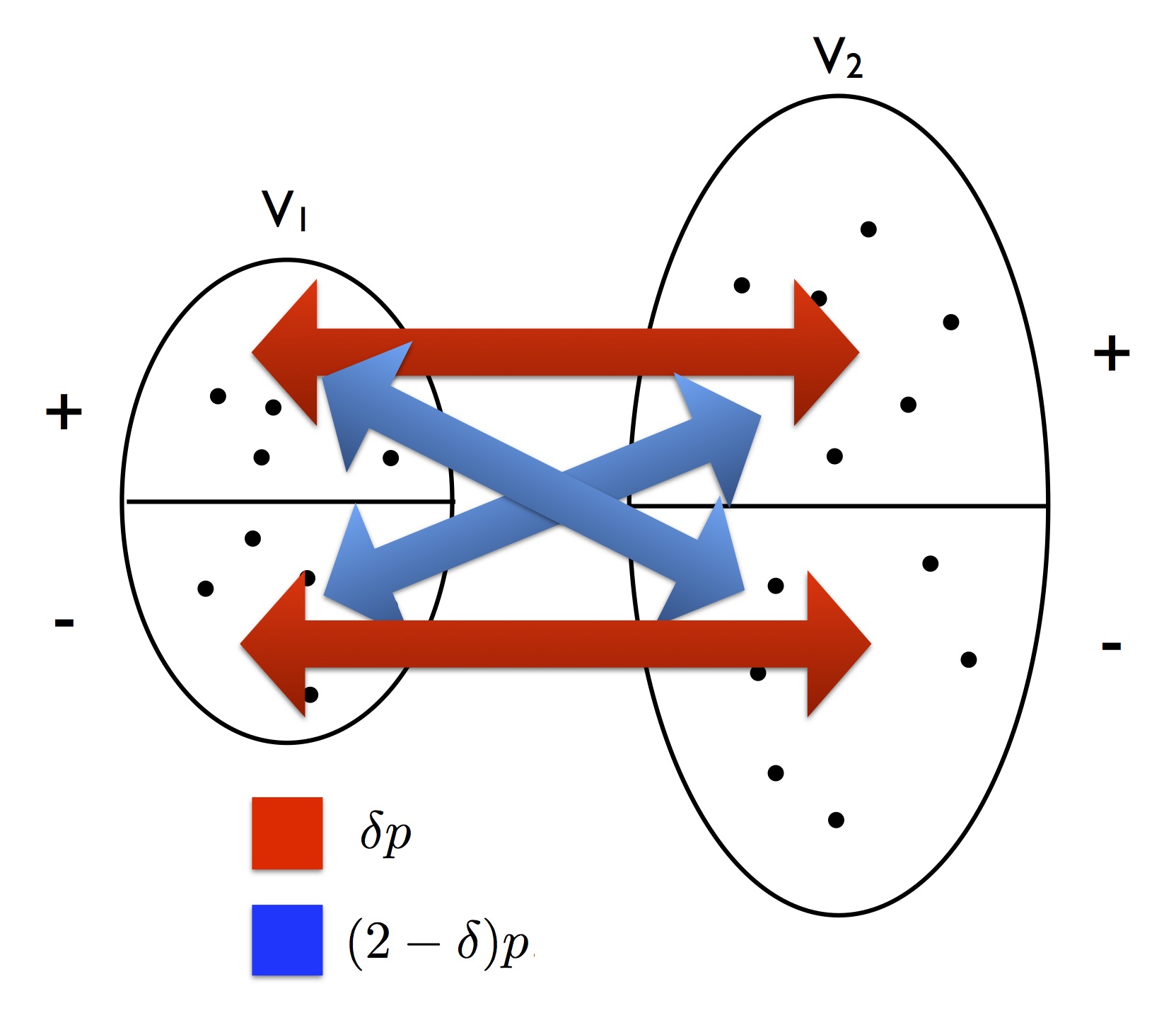}
\caption{Bipartite stochastic block model on $V_1$ and $V_2$. Red edges are added with probability $\del p$ and blue edges are added with probability $(2-\del)p$.}
\end{figure}

\paragraph{Preliminaries and assumptions}

In the application to planted satisfiability, it suffices to recover $\sigma$, the partition of the smaller vertex set, $V_1$, and so we focus on that task here; we will accomplish that task even when the number of edges is much smaller than the size of $V_2$.  For a planted $k$-SAT problem or $k$-uniform hypergraph partitioning problem on $n$ variables or vertices, the reduction gives vertex sets of size $n_1 = \Theta(n), n_2 = \Theta(n^{k-1})$, and so the relevant cases are extremely unbalanced.

We will say that an algorithm \textit{detects} the partition if for some fixed $\eps>0$, independent of $n_1$, whp it returns an $\eps$-correlated partition, i.e. a partition that agrees with $\sigma$ on a ($1/2 + \eps$)-fraction of vertices in $V_1$ (again, up to the sign of $\sigma$).

We will say an algorithm \textit{recovers} the partition of $V_1$ if whp the algorithm returns a partition that agrees with $\sigma$ on $1- o(1)$ fraction of vertices in $V_1$. Note that agreement is up to sign as $\sigma $ and $-\sigma$ give the same partition.

\subsection{Optimal algorithms for detection}

On the basis of heuristic analysis of the belief propagation algorithm, \cite{decelle2011asymptotic} made the striking conjecture that in the two part stochastic block model, with interior edge probability $a/n$, crossing edge probability $b/n$, there is a \textit{sharp threshold} for detection: for $(a-b)^2 > 2(a+b)$ detection can be achieved with an efficient algorithm, while for $(a-b)^2 \le 2(a+b)$, detection is impossible for any algorithm.  This conjecture was proved by \cite{mossel2012stochastic, mossel2013proof} and \cite{massoulie2014community}.

Our first result is an analogous sharp threshold for detection in the bipartite stochastic block model at $p = (\del-1)^{-2} (n_1 n_2)^{-1/2}$, with an algorithm based on a reduction to the SBM, and a lower bound based on a connection with the non-reconstruction of a broadcast process on a tree associated to a two-type Galton Watson branching process (analogous to the proof for the SBM \cite{mossel2012stochastic} which used a single-type Galton Watson process).

\begin{figure}[h!]
\begin{center}
\fbox{
\parbox{\textwidth}{
{\bf Algorithm: SBM Reduction.}
\begin{enumerate}
\item Construct a graph $G^\prime$ on the vertex set $V_1$ by joining $u$ and $w$ if they are both connected to the same vertex $v \in V_2$ and $v$ has degree exactly $2$.
\item  Randomly sparsify the graph (as detailed in Section \ref{sec:detection}).
\item  Apply an optimal algorithm for detection in the SBM from \cite{massoulie2014community, mossel2013proof, bordenave2015non} to partition $V_1$.
\end{enumerate}
}
}
\end{center}
\end{figure}


\begin{thm}
\label{thm:detection}
Let  $\delta  \in [0,2] \setminus \{1\} $ be fixed and $n_2 = \omega( n_1)$.  Then there is a polynomial-time algorithm that detects the partition $V_1 = A_1 \cup B_1$ whp if 
\[ p >  \frac{1 + \eps}{(\del-1)^2 \sqrt{n_1 n_2} } \]
for any fixed $\eps > 0$. 
\end{thm}

\begin{thm}
\label{thm:noDetect} 
 On the other hand, if $n_2 \ge n_1$ and  
 \[ p \le  \frac{1}{(\del-1)^2 \sqrt{n_1 n_2} }, \]
 then no algorithm can detect the partition whp.
 \end{thm}

Note that for $p \le \frac{1}{\sqrt{n_1 n_2}}$ it is clear that detection is impossible: whp there is no giant component in the graph.  The content of Theorem~\ref{thm:noDetect} is finding the sharp dependence on $\del$.

\subsection{Spectral algorithms}
One common approach to graph partitioning is spectral: compute eigenvectors or singular vectors of an appropriate matrix and round the vector(s) to partition the vertex set.  In our setting, we can take the $n_1 \times n_2$ rectangular biadjacency matrix $M$, with rows and columns indexed by the vertices of $V_1$ and $V_2$ respectively, with a $1$ in the entry $(u,v)$ if the edge $(u,v)$ is present, and a $0$ otherwise. The matrix $M$ has independent entries that are $1$ with probability $\del p$ or $(2 - \del) p$ depending on the label of $u$ and $v$ and $0$ otherwise.

\begin{figure}[h!]
\begin{center}
\fbox{
\parbox{\textwidth}{
{\bf Algorithm:  Singular Value Decomposition.}
\begin{enumerate}
\item Compute the left singular vector of $ M$ corresponding to the second largest singular value.
\item  Round the singular vector to a vector $z \in \{ \pm 1\}^{n_1}$ by taking the sign of each entry. 
\end{enumerate}
}
}
\end{center}
\end{figure}


   
A typical analysis of spectral algorithms requires that the second largest eigenvalue or singular value of  the expectation matrix $\E M$ is much larger than the spectral norm of the noise matrix, $( M - \E M)$.  But here we have $\| M - \E M \| = \tilde \Theta ( \sqrt{ n_2 p}  )$, which is in fact much larger than $\lam_2 ( \E M) =  \Theta ( p \sqrt{n_1 n_2} )$ when $p = o ( n_1^{-1})$.  Does this doom the spectral approach at lower densities?

\begin{question}
\label{q:main}
For what values of $p = p(n_1, n_2)$ is the singular value decomposition (SVD) of $ M$ correlated with the vector $\sigma$ indicating the partition of $V_1$? 
\end{question}

In particular, this question was asked by \cite{feldman2014algorithm}. We show that there are two thresholds, both well below $p = n_1^{-1}$: at $p = \tilde \Omega ( n_1^{-2/3} n_2^{-1/3})$ the second singular vector of $ M$ is correlated with the partition of $V_1$, but below this density, it is uncorrelated with the partition, and in fact localized. Nevertheless, we give a simple spectral algorithm based on modifications of $M$ that matches the bound $p = \tilde O( (n_1 n_2)^{-1/2})$ achieved with subsampling by \cite{feldman2014algorithm}.  In the case of very unbalanced sizes, in particular in the applications noted above, these thresholds can differ by a polynomial factor in $n_1$.  

\begin{figure}[h!]
\begin{center}
\fbox{
\parbox{\textwidth}{
{\bf Algorithm: Diagonal Deletion SVD.}
\begin{enumerate}
\item Let $  B = M M^T - \text{diag}(MM^T)$ (set the diagonal entries of $MM^T$ to $0$).
\item  Compute the second eigenvector of $ B $.
\item  Round the  eigenvector to a vector $z \in \{ \pm 1\}^{n_1}$ by taking the sign of each entry.
\end{enumerate}
}
}
\end{center}
\end{figure}



Our results locate two different thresholds for spectral algorithms for the bipartite block model: while the usual SVD is only effective with $p = \tilde \Omega( n_1 ^{-2/3} n_2^{-1/3})$, the modified diagonal deletion algorithm is effective already at $p = \tilde \Omega( n_1 ^{-1/2} n_2^{-1/2})$, which is optimal up to logarithmic factors. In particular, when $n_1 =n, n_2=n^{k-1}$ for some $k \ge 3$, as in the application above, these thresholds are separated by a polynomial factor in $n$.

\begin{figure}[h!]
\centering
\includegraphics[scale=0.45]{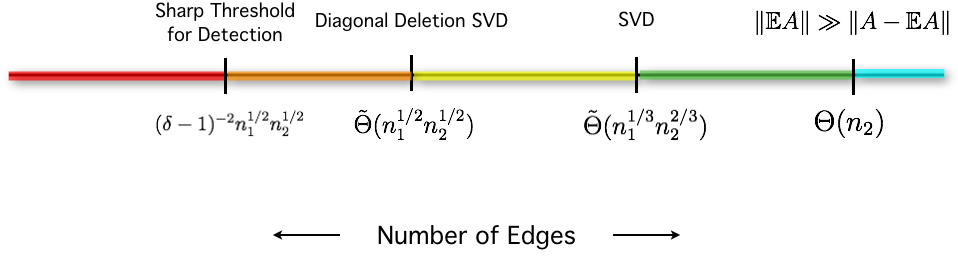}
\caption{Main theorems illustrated.}
\end{figure}

First we give positive results for recovery using the two spectral algorithms.

\begin{thm}
\label{thm:main}
Let $n_2 \ge n_1 \log^4 n_1$, with $n_1 \to \infty$. Let $\del \in [0,2] \setminus \{1\}$ be fixed with respect to $n_1,n_2$. Then there exists a universal constant $C>0$ so that
\begin{enumerate} 
\item If $p = C (n_1 n_2)^{-1/2} \log n_1$, then whp the diagonal deletion SVD algorithm recovers the partition $V_1 =  A_1 \cup B_1$.  
\item If $p = C n_1^{-2/3} n_2^{-1/3} \log n_1$, then whp the unmodified SVD algorithm recovers the partition.  
\end{enumerate}

\end{thm}

Next we show that below the recovery threshold for the SVD, the top left singular vectors are in fact \textit{localized}: they have nearly all of their mass on a vanishingly small fraction of coordinates.  

\begin{thm}
\label{thm:SVDfail}
\item Let $n_2 \ge n_1 \log^4 n_1$. For any constant $c>0$, let $p =c n_1^{-2/3} n_2^{-1/3}$, $t \le  n_1^{1/3}$, and $r= n_1/\log n_1$. Let $\overline \sigma = \sigma/ \sqrt{n_1}$,  and $v_1, v_2, \dots v_t$ be the top $t$ left unit-norm singular vectors of $M$. 

Then, whp, there exists a set $S \subset \{1, \dots, n_1 \}$ of coordinates, $|S| \le r$, so that for all $1 \le i \le t$, there exists a unit vector $u_i$ supported on $S$ so that 
\[ \| v_i - u_i \| = o(1). \]
That is, each of the first $t$ singular vectors has nearly all of its weight on the coordinates in $S$.  In particular, this implies that for all $1 \le i \le t$, $v_i$ is asymptotically uncorrelated with the planted partition: 
\[ | \overline \sigma \cdot v_i | = o(1). \]
\end{thm}

One point of interest in Theorem \ref{thm:SVDfail} is that in this case of a random biadjacency matrix of unbalanced dimension, the localization and delocalization of the singular vectors can be understood and analyzed in a simple manner, in contrast to the more delicate phenomenon for random square adjacency matrices.

Our techniques use bounds on the norms of random matrices and eigenvector perturbation theorems, applied to carefully chosen decompositions of the matrices of interest.  In particular, our proof technique suggested the Diagonal Deletion SVD, which proved much more effective than the usual SVD algorithm on these unbalanced bipartite block models, and has the advantage over more sophisticated approaches of being extremely simple to describe and implement.  We believe it may prove effective in many other settings.


Under what conditions might we expect the Diagonal Deletion SVD outperform the usual SVD?  The SVD is a central algorithm in statistics, machine learning, and computer science, and so any general improvement would be useful.  The bipartite block model addressed here has two distinctive characteristics: the dimensions of the matrix $M$ are extremely unbalanced, and the entries are very sparse Bernoulli random variables, a distribution whose fourth moment is much larger than the square of its second moment.  These two facts together lead to the phenomenon of multiple spectral thresholds and the outperformance of the SVD by the Diagonal Deletion SVD.  Under both of these conditions we expect dramatic improvement by using diagonal deletion, while under  one or the other condition, we expect mild improvement.  We expect diagonal deletion will be effective in the more general setting of recovering a low-rank matrix in the presence of random noise, beyond our setting of adjacency matrices of graphs.

\section{Planted $k$-SAT and hypergraph partitioning}
\label{sec:kSAT}

 \cite{feldman2014algorithm} reduce three planted problems  to solving the bipartite block model: planted hypergraph partitioning, planted random $k$-SAT, and Goldreich's planted CSP.  
We describe the reduction here and calculate the density at which our algorithm can detect the planted solution by solving the resulting bipartite block model.  

We state the general model in terms of hypergraph partitioning first.  

\paragraph{Planted hypergraph partitioning}
Fix a function $Q : \{ \pm 1 \}^k \to [0,1]$ so that $\sum_{x \in \{\pm 1\}^k} Q(x) =1$. Fix parameters $n$ and  $p \in (0,1)$ so that $\max_x Q(x) 2^k p \le 1$.  Then we define the planted $k$-uniform hypergraph partitioning model as follows: 
\begin{itemize}
\item Take a vertex set $V$ of size $n$.
\item Assign labels `+' and `-' independently with probability $1/2$ to each vertex in $V$.  Let $\sigma \in \{ \pm 1 \}^{n}$ denote the labels of the vertices. 
\item Add (ordered) $k$-uniform hyperedges independently at random according to the distribution 
\[ \Pr(e) = 2^k p \cdot Q (\sigma(e)) \]
where $\sigma(e)$ is the evaluation of $\sigma $ on the vertices in $e$.
\end{itemize}
\noindent
\textbf{Algorithmic task:} Determine the labels of the vertices given the hypergraph, and do so with an efficient algorithm at the smallest possible edge density $p$. 

Usually $Q$ will be symmetric in the sense that $Q(x)$ depends only on the number of $+1$'s in the vector $x$, and in this case we can view hyperedges as unordered. We assume that $Q$ is not identically $2^{-k}$ as this distribution would simply be uniform and the planted partition would not be evident. 

Planted $k$-satisfiability is defined similarly: we fix an assignment $\sigma $ to $n$ boolean variables which induces a partition of the set of $2n$ literals (boolean variables and their negations) into true and false literals.  Then we add $k$-clauses independently at random, with probability proportional to the evaluation of $Q$ on the $k$ literals of the clause.

Planting distributions for the above problems are classified by their \textit{distribution complexity}, $r = \min_{S \ne \emptyset} \{ |S| : \hat Q(S) \ne 0 \}$, where $\hat Q(S)$ is the discrete Fourier coefficient of $Q$ corresponding to the subset $S \subseteq [k]$.  This is an integer between $1$ and $k$, where $k$ is the uniformity of the hyperedges or clauses.

A consequence of Theorem~\ref{thm:detection} is the following:

\begin{thm}
\label{thm:ksat}
There is an efficient algorithm to detect the planted partition in the random $k$-uniform hypergraph partitioning problem, with planting function $Q$, when
\[ p> (1+\eps)\min_{S \subseteq [k]}   \frac{1 }{ 2^{2k} \hat Q(S)^2 n^{k-|S|/2} }    \]
for any fixed $\eps >0$. 
Similarly, in the planted $k$-satisfiability model with planting function $Q$, there is an efficient algorithm to detect the planted assignment when
\[ p> (1+\eps)\min_{S \subseteq [k]}   \frac{1 }{ 2^{2k} \hat Q(S)^2 (2n)^{k-|S|/2} }   . \]
In both cases, if the distribution complexity of $Q$ is at least $3$, we can achieve full recovery at the given density.
\end{thm}

\begin{proof}
Suppose $Q$ has distribution complexity $r$. Fix a set $S \subseteq [k]$ with $\hat Q(S) \ne 0$, and $|S|=r$.  The first step of the reduction of \cite{feldman2014algorithm} transforms each $k$-uniform hyperedge into an $r$-uniform hyperedge by selecting the vertices indicated by the set $S$.  Then a bipartite block model is constructed on vertex sets $V_1, V_2$, with $V_1$ the set of all vertices in the hypergraph (or literals in the formula), and $V_2$ the set of all $(r-1)$-tuples of vertices or literals.  An edge is added by taking each $r$-uniform edge and splitting it randomly into sets of size $1$ and $r-1$ and joining the associated vertices in $V_1$ and $V_2$. The parameters in our model are $n_1 =n$ and $n_2 \sim n^{r-1}$ (considering  ordered $(r-1)$-tuples of vertices or literals).

These edges appear with probabilities that depend on the parity of the number of vertices on one side of the original partition in the joined sets, exactly the bipartite block model addressed in this paper; the parameter $\del$ in the model is given by $\del = 1+ 2^k \hat Q(S)$ (see Lemma 1 of \cite{feldman2014algorithm}).  Theorems \ref{thm:detection}  then states that detection in the resulting block model exhibits a sharp threshold at edge density $p^*$ , with $p^* =   \frac{1}{2^{2k} \hat Q(S)^2 n^{k-r/2}}$. The difference in bounds in Theorem~\ref{thm:ksat} is due to the two models having $n$ vertices and $2n$ literals respectively.


To go from an $\eps$-correlated partition to full recovery, if $r \ge 3$, we can appeal to Theorem 2 of  \cite{bogdanov2009security} and achieve full recovery using only a linear number of additional hyperedges or clauses, which is lower order than the $\Theta(n^{r/2})$ used by our algorithm. 
\end{proof}
Note that Theorem~\ref{thm:noDetect} says that no further improvement can be gained by analyzing this particular reduction to a bipartite stochastic block model.


There is some evidence that up to constant factors in the clause or hyperedge density, there may be no better efficient algorithms \cite{ow, feldman2013complexity}, unless the constraints induce a consistent system of linear equations.  But in the spirit of \cite{decelle2011asymptotic}, we can ask if there is in fact a sharp threshold for detection of planted solutions in these models.  In one special case, such sharp thresholds have been conjectured:  \cite{krzakala2014reweighted} have conjectured threshold densities based on fixed points of belief propagation equations.  The planted $k$-SAT distributions covered, however, are only those with distribution complexity $r=2$: those that are known to be solvable with a linear number of clauses.  We ask if there are sharp thresholds for detection in the general case, and in particular for those distributions with distribution complexity $r \ge 3$ that cannot be solved by Gaussian elimination. In particular, in the case of the parity distribution we conjecture that there is a sharp threshold for detection.  

\begin{conj}
Partition a set of $n$ vertices at random into sets $A, B$.  Add $k$-uniform hyperedges independently at random with probability $\del p$ if the number of vertices in the edge from $A$ is even and $(2- \del)p$ if the number of vertices from $A$ is odd.  Then for any $\del \in (0,2)$ there is a constant $c_{\del}$ so that $p = c_\del n^{-k/2}$ is a sharp threshold for detection of the planted partition by an efficient algorithm.  That is, if $p> (1+\eps)c_\del n^{-k/2}$, then there is a polynomial-time algorithm that detects the partition whp, and if $p \le c_\del n^{-k/2}$ then no polynomial-time algorithm can detect the partition whp. 
\end{conj}

This is a generalization to hypergraphs of the SBM conjecture of \cite{decelle2011asymptotic}; the $k=2$ parity distribution is that of the stochastic block model.  We do not venture a guess as to the precise constant $c_\del$, but even a heuristic as to what the constant might be would be very interesting. 

\subsection{Relation to Goldreich's generator}
\label{sec:goldreich}
\cite{goldreich2000candidate}'s pseudorandom generator or one-way function can be viewed as a variant of planted satisfiability.  Fix an assignment $\sigma$ to $n$ boolean variables, and fix a predicate $P : \{ \pm 1 \}^k \to \{0,1 \}$.  Now choose $m$ $k$-tuples of variables \textit{uniformly} at random, and label the $k$-tuple with the evaluation of $P$ on the tuple with the boolean values given by $\sigma$.  In essence this generates  a uniformly random $k$-uniform hypergraph with labels that depend on the planted assignment and the fixed predicate $P$.  The task is to recover $\sigma$ given this labeled hypergraph.  The algorithm we describe above will work in this setting by simply discarding all hyperedges labeled $0$ and working with the remaining hypergraph.  
 

\section{Related work}

The stochastic block model has been a source of considerable recent interest. There are many algorithmic approaches to the problem, including algorithms based on maximum-likelihood methods \cite{sn}, belief propagation \cite{decelle2011asymptotic}, spectral methods \cite{mcsherry2001spectral}, modularity maximization \cite{bc}, and combinatorial methods \cite{bcls}, \cite{df}, \cite{js}, \cite{ck}.  \cite{coja2010graph} gave the first algorithm to detect partitions in the sparse, constant average degree regime. \cite{decelle2011asymptotic} conjectured the precise achievable constant and subsequent algorithms \cite{massoulie2014community, mossel2013proof, bordenave2015non, as} achieved this bound. Sharp thresholds for full recovery (as opposed to detection) have been found by \cite{mossel2015consistency, abh, hajek2015achieving}.     

 
 \cite{bogdanov2009security} used ideas for reconstructing assignments to random $3$-SAT formulas in the planted $3$-SAT model to show that Goldreich's construction of a one-way function in \cite{goldreich2000candidate} is not secure when the predicate correlates with either one or two of its inputs. For more on Goldreich's PRG from a cryptographic perspective see the survey of \cite{applebaum2013cryptographic}. 
 
  \cite{feldman2014algorithm} gave an algorithm to recover the partition of $V_1$ in the bipartite stochastic block model to solve instances of planted random $k$-SAT and planted hypergraph partitioning using subsampled power iteration.

A key part of our analysis relies on looking at an auxiliary graph on $V_1$ with edges between vertices which share a common neighbor; this is known as the one-mode projection of a bipartite graph:  \cite{zhou2007bipartite} give an approach to recommendation systems using a weighted version of the one-mode projection.  One-mode projections are implicitly used in studying collaboration networks, for example in \cite{newman2001scientific}'s analysis of scientific collaboration networks.  \cite{larremore2014efficiently} defined a general model of bipartite block models, and propose a community detection algorithm that does not use one-mode projection.

The behavior of the singular vectors of a low rank rectangular matrix plus a noise matrix was studied by  \cite{benaych2012singular}.  The setting there is different: the ratio between $n_1$ and $n_2$ converges, and the entries of the noise matrix are mean $0$ variance $1$.  

\cite{butucea2015sharp} and \cite{hajek2015submatrix} both consider the case of recovering a planted submatrix with elevated mean in a random rectangular Gaussian matrix. 



 \paragraph{Notation}
 
 All asymptotics are as $n_1 \to \infty$, so for example, `$E$ occurs whp' means $\displaystyle\lim_{n_1 \to \infty} \Pr(E) = 1$.   We write $f(n_1) = \tilde O(g(n_1))$ and $f(n_1) = \tilde \Omega(g(n_1))$ if there exist constants $C,c$ so that $f(n_1) \le C \log^c (n_1) \cdot g(n_1)$ and $f(n_1) \ge  g(n_1)/ (C \log^c(n_1))$ respectively.  For a vector, $\| v \|$ denotes the $l_2$ norm.  For a matrix, $\| A \|$ denotes the spectral norm, i.e. the largest singular value (or largest eigenvalue in absolute value for a square matrix). For ease of reading, $C$ will always denote an absolute constant, but the value may change during the course of the proofs.

\section{Proof of Theorem~\ref{thm:detection}: detection}
\label{sec:detection}

In this section we prove Theorem~\ref{thm:detection}, giving an optimal algorithm for detection in the bipartite stochastic block model when $n_2 = \omega(n_1)$.  The main idea of the proof is that almost all of the information in the bipartite block model is in the subgraph induced by $V_1$ and the vertices of degree two in $V_2$.  From this induced subgraph of the bipartite graph we form a graph $G^\prime$ on $V_1$ by replacing each path of length two from $V_1$ to $V_2$ back to $V_1$
with a single edge between the two endpoints in $V_1$.  We then apply an algorithm from \cite{massoulie2014community, mossel2013proof}, or \cite{bordenave2015non} to detect the partition.  

\begin{proof}[Proof of Theorem~\ref{thm:detection}]
Fix $\eps >0$. Given an instance $G$ of the bipartite block model with 
\[ p =(1+\eps) (\del-1)^{-2} (n_1 n_2)^{-1/2}, \]
 we reduce to a graph $G^\prime$ on $V_1$ as follows: 
\begin{itemize} 
\item Sort $V_2$ according to degrees and remove any vertices (with their accompanying edges) which are not of degree $2$.
 \item We now have a union of $2$-edge paths from vertices in $V_1$ to vertices in $V_2$ and back to vertices in $V_1$.  Create a multi-set of edges $\mathcal E$ on $V_1$ by replacing each $2$-path $u - v -w$ by the edge $(u,w)$.   
 \item Choose $N$ from the distribution $\text{Poisson}((1+\eps) (\del-1)^{-4} n_1/2)$. 
\item If $N > | \mathcal E |$, then stop and output `failure'.  Otherwise, select $N$ edges uniformly at random from $\mathcal E$ to form the graph $G^\prime$ on $V_1$, replacing any edge of multiplicity greater than one with a single edge.   
 \item Apply an SBM  algorithm to $G^\prime$ to partition $V_1$. 
 \end{itemize}

We now determine the distribution of $G^\prime$ conditioned on $\sigma$.  Let $\beta_1$ be the bias of  $+1$ labels in $\sigma$, $\beta_1 = \frac{1}{n_1} \sum_{u \in V_1} \sigma(u)$.  Conditioned on $\beta_1$, the degrees $d(v_1), \dots , d(v_{n_2})$ of the vertices of $V_2$ are independent, identically distributed random variables. Let $Y_i =d(v_i)$.  Under the high probability event that $\beta_1 = o(n_1^{-1/3})$, we can compute 
\begin{align*}
\Pr[Y_i =2 | \sigma, \beta_1 = o(n_1^{-1/3})] &= \frac{n_1^2 p^2}{2} (1+o(1))
\end{align*}
and so whp, $| \mathcal E | = \frac{(1+\eps)^2 n_1 }{ (\del-1)^4   } (1+o(1))  $, and $N < | \mathcal E|$.  Note that it is only at this step that we require the assumption that $n_2 = \omega(n_1)$. 

Conditioned on $\sigma$, the edges in $\mathcal E$ are independent and identically distributed, with distribution of a given edge $e=(u,v)$ as
\begin{align*}
p_{a} &=\Pr[e = (u,v) | \sigma(u)=\sigma(v) ]  \\
&= \frac{\frac{1}{2} \del^2 + \frac{1}{2} (2-\del)^2   }{ \left( \binom  {(\beta_1+1)n_1/2}{2} + \binom  {(\beta_1-1)n_1/2}{2} \right ) \frac{ \del^2 +(2-\del)^2}{2} + (\beta_1+1)(\beta_1-1)(n_1^2/4) \del (2-\del)   }   \\
 p_{b} &= \Pr[e = (u,v) | \sigma(u) \ne \sigma(v)] \\
 &= \frac{\del (2-\del)  }{ \left( \binom  {(\beta_1+1)n_1/2}{2} + \binom  {(\beta_1-1)n_1/2}{2} \right ) \frac{ \del^2 +(2-\del)^2}{2} + (\beta_1+1)(\beta_1-1)(n_1^2/4) \del (2-\del)   } 
\end{align*}
When $\beta_1 = o(n_1^{-1/3})$, 
\begin{align*}
p_{a} &=  \frac{2 - 2\del +\del^2   }{  2 } \cdot \frac{4}{n_1^2} (1+o(1)) \,\, \text{ and } \,\, p_{b} = \frac{2\del -\del^2  } {2 }\cdot \frac{4}{n_1^2} (1+o(1)).
\end{align*}

By Poisson thinning this means that the number of times each $++$ or $--$ edge appears in the subsampled collection of edges is a Poisson of mean $p_{a} \E N$,  and each $+-$ edge according to a Poisson of mean $p_{b} \E N$, and all of these edge counts are independent. 

Now define $a$ so that $\Pr [\text{Poisson} (  p_{a} \cdot \E N) \ge 1] = \frac{a}{n_1}$, and $b$ so that $\Pr [\text{Poisson} (  p_{b} \cdot \E N) \ge 1] = \frac{b}{n_1}$.  

From the construction above,  conditioned on $\sigma$ the distribution of $G^\prime$ is that of the stochastic block model on $V_1$ with partition $\sigma$: each edge interior to the partition is present with probability $a/n_1$, each crossing edge with probability $b/n_1$, and all edges are independent. 

For $\sigma$ such that $\beta_1 = o(n^{-1/3})$, we have
\begin{align*}
a &=  \frac{(1+\eps) (2- 2\del + \del^2)  }{ (\del-1)^4   }  (1+ o(1))\\
b & =  \frac{(1+\eps) (2\del - \del^2)  }{ (\del-1)^4   }  (1+ o(1))
\end{align*}
For these values of $a$ and $b$  the condition for detection in the SBM, $(a-b)^2 \ge (1+\eps) 2 (a+b)$ is satisfied and so whp the algorithms from \cite{massoulie2014community, mossel2013proof, bordenave2015non} will find a partition that agrees with $\sigma $ on $1/2 + \eps^\prime$  fraction of vertices.  
\end{proof}

\section{Proof of Theorem~\ref{thm:noDetect}: impossibility}
The proof of impossibility below the threshold $(a-b)^2 = 2 (a+b)$ in \cite{mossel2012stochastic} proceeds by showing that the $\log n$ depth neighborhood of a vertex $\rho$, along with the accompanying labels, can be coupled to a binary symmetric broadcast model on a Poisson Galton-Watson tree.  In this model, it was shown by \cite{evans2000broadcasting} that reconstruction, recovering the label of the root given the labels at depth $R$ of the tree, is impossible as $R \to \infty$, for the corresponding parameter values (the critical case was shown by \cite{pemantle2010critical}).  

In the binary symmetric broadcast model, the root of a tree is labeled with a uniformly random label $+1$ or $-1$, and then each child takes its parent's label with probability $1- \eta$ and the opposite label with probability $\eta$, independently over all of the parent's children.  The process  continues in each successive generation of the tree.

The criteria for non-reconstruction can be stated as $(1- 2\eta)^2 B \le 1$, where $B$ is the branching number of the tree $T$.  The branching number is $B = p_c(T)^{-1}$, where $p_c$ is the critical probability for bond percolation on $T$ (see \cite{lyons1990random} for more on the branching number).

Assume first that $n_2 \sim c n_1$ for some constant $c$, and that $p = d/n_1$. Then there is a natural multitype Poisson branching process that we can associate to the bipartite block model: nodes of type $1$, corresponding to vertices in $V_1$, have a $\pois(cd)$ number of children of type $2$; nodes of type $2$, corresponding to vertices in $V_2$, have a $\pois(d)$ number of children of type $1$.  The branching number of this distribution on trees is $\sqrt{c} \cdot d$, an easy calculation by reducing to a one-type Galton Watson process by combining two generations into one.  Transferring the block model labeling to the branching process gives $\eta = \del /2$, and so the threshold for reconstruction is given by

\[ (\del-1)^2 \sqrt{c} d \le  1  \]
or in other words,
\[ p \le  \frac{1}{(\del-1)^2 \sqrt{n_1 n_2}} \]
exactly the threshold in Theorem \ref{thm:noDetect}.  In fact, in this case the proof from \cite{mossel2012stochastic} can be carried out in essentially the exact same way in our setting.  

Now  take $n_2 = \omega(n_1)$.  A complication arises: the distribution of the number of neighbors of a  node of type $1$ does not converge (its mean is $n_2 p \to \infty$), and the distribution of the number of neighbors of a node of type $2$ converges to a delta mass at $0$.  But this can be fixed by ignoring the vertices in $V_2$ of degree $0$ and $1$.  Now we explore from a vertex $\rho \in V_1$, but discard any vertices from $V_2$ that do not have a second neighbor.  We denote by $\hat G$ the subgraph of $G$ induced by $V_1$ and the vertices of $V_2$ of degree at least $2$.  Let $T$ be the branching process associated to this modified graph: nodes of type $1$ have $\pois(d^2)$ neighbors of type $2$, and nodes of type $2$ have exactly $1$ neighbor of type $1$, where here $p = d/\sqrt{n_1 n_2}$.  The branching number of  this  process is $d$, and the reconstruction threshold is $(\del -1)^2 d \le 1$, again giving the threshold $p \le  \frac{1}{(\del-1)^2 \sqrt{n_1 n_2}}$, as required. 

As in \cite{mossel2012stochastic}, the proof of impossibility will show the stronger statement that conditioned on the label of a fixed vertex $w \in V_1$ and the graph $G$, the variance of the label of another fixed vertex $\rho$ tends to $1$ as $n_1 \to \infty$.  The proof of this fact has two main ingredients: showing that the depth $R$ neighborhood of a vertex $\rho$ in the bipartite block model (with vertices of degree $0$ and $1$ in $V_2$ removed) can be coupled with the branching process described above, and showing that conditioned on the labels on the boundary of the neighborhood, the label of $\rho$ is asymptotically independent of the rest of the graph and the labels outside of the neighborhood.  We will use the notation from Section 4 of \cite{mossel2012stochastic} and indicate the places in which our proof must differ; the most significant is that we must show that the vertices of degree $0$ and $1$ in $V_2$ give essentially no information about the label of $\rho$.

First note that in Proposition 4.2 from \cite{mossel2012stochastic}, $R = \Theta(\log n)$, but for the proof of Theorem 2.1 all that is required is $R = \omega(1)$. We  choose 
\[R=  \frac{1}{20 d} \cdot \min \{ \log n_1, \log (n_2/n_1)  \} = \omega(1) .\]

Let $T$ be the branching process described above, starting with a root of type $1$.  We will denote the labeling functions of nodes of type $1$ and type $2$ in $T$ by $\hat \sigma$ and $\hat \tau$ respectively. We will consider two steps of the exploration process at once, so the depth $0$ neighborhood is $\rho$ itself, the depth $1$ neighborhood is $\rho$, its neighbors (of degree at least $2$), and the neighbors of these neighbors.  The depth $r$ neighborhood then includes those vertices in $V_1$ at distance $2r$ from $\rho$. Let $\hat G_r$ be this depth $r$ neighborhood in $\hat G$, and $\sigma_r, \tau_r$ be the labelings of $V_1$ and $V_2$ restricted to the vertices in $\hat G_r$.  Define $T_R, \hat \sigma, \hat \tau$ as the same objects for the tree process. Let $\partial^{1} \hat G_r, \partial ^1 T_r$ be the set of vertices from $V_1$, nodes of type $1$, in the last layer, and $\partial^{2}  \hat G_r, \partial ^2 T_r$ the vertices from $V_2$ and nodes of type $2$ in the last layer. Let $\overline V_r = V_1 \cup V_2 \setminus V(\hat G_r)$.  We will show:

\begin{lemma}
\label{lem:Rdepth}
With $R$ as above, there is a coupling so that $(\hat G_R, \sigma_R,\tau_R) = (T_R, \hat \sigma_R, \hat \tau_R)$ whp.   
\end{lemma}

\begin{proof}
$T$ can be constructed by three sequences of independent random variables $Y_u^a \sim \pois(d^2\del/2 )$,  $Y_u^b \sim \pois(d^2 (2 -\del)/2)$, for $u$ of type $1$ in $T$, and $X_v \sim \text{Bernoulli}(\del/2)$ for $v$ of type $2$ in $T$.  To create the branching process, we start with the root $\rho$ of type $1$, and assign it $+1$ or $-1$ label at random.  We then assign it $Y_\rho^a$ type-$2$ children of the same label and $Y_\rho^b$ type-$2$ children of the opposite label. All together the number of children has a $\pois(d^2)$ distribution, and the labels are selected independently to agree with $\rho$ with probability $\del /2$ and to disagree with probability $1- \del /2$.  Now each child $v$ of type $2$ has exactly one child of its own, whose label agrees if $X_v =1$ and disagrees otherwise. Then the process continues inductively. 

Now consider exploring the depth $R$ neighborhood of $\rho$ in $\hat G$. We index the vertices from $V_1$ in the order in which we encounter them in this breadth-first exploration.    We explore two layers of the neighborhood at once: the active vertex $u_i$ will always be from $V_1$.  To explore from $u_i$, we reveal all edges from $u_i$ to unexplored vertices in $V_2$; call these neighbors $N(u_i)$.  We then set all $v \in N(u_i)$ to be explored, and query all edges from $N(u_i)$ to unexplored vertices in $V_1$; call these vertices $N^2(u_i)$, as they are all connected by a path of length $2$ to $u_i$ in $\hat G$. Set all vertices in $N^2(u_i)$ to explored, and place them in a FIFO queue of vertices.  Then set $u_i$ to dead, and take the next vertex $u_{i+1}$ from the queue, set to active, and repeat.    

Let $N_{aa}(u_i), N_{ab}(u_i), N_{ba}(u_i), N_{bb}(u_i)$ be the number of paths of length $2$ from $u_i$ to an unexplored vertex in $V_1$, with the subscripts denoting whether the labels along the path agree or disagree with the label of $u_i$; e.g. if $\sigma (u_i) =+1$, then $N_{ba}(u_i)$ is the number of paths that go through a vertex $v \in V_2$ with label $-1$ and then to an unexplored vertex $w \in V_1$ with label $+1$. Let $V_1^+(i), V_1^-(i)$ be the number of unexplored vertices in $V_1$ with the respective labels at the moment $u_i$ becomes active, and likewise for $V_2^+(i), V_2^-(i)$.   If we condition on the bias of $\sigma$ and $\tau$, $\displaystyle \beta_1 = \frac{1}{n_1} \sum_{u \in V_1} \sigma(u)$ and $\displaystyle \beta_2 = \frac{1}{n_2} \sum_{v \in V_2} \tau(v)$, then at each step of the exploration, the distribution of the $N_{**}$'s depends only on the $V_i^*$'s. 

As in \cite{mossel2012stochastic}, let $A_r$ be the event that no vertex in $\overline V_{r-1}$ has more than one neighbor in $\hat G_{r-1}$; let  $B_r$ be the event that no vertex in $\partial ^2 \hat G_r$ has more than one neighbor in $\partial ^1 \hat G_r$, also define the event $D_r = \{ d(v) \le 2 \text{ for all } v \in \partial^2 \hat G_r   \}$.  Then analogously to Lemma 4.3 in \cite{mossel2012stochastic}, we have
\begin{lemma}
\label{lem43}
If 
\begin{enumerate}
\item $( \hat G_{r-1}, \sigma_{r-1}, \tau_{r-1}) = (T_{r-1}, \hat \sigma_{r-1}, \hat \tau_{r-1})$.
\item For every $u \in \partial^1 \hat G_{r-1}$, $N_{aa}(u) + N_{ab}(u)= Y_u^a = \emph{\pois}(d^2\delta/2)$; $N_{ba}(u) + N_{bb}(u)= Y_u^b = \emph{\pois}(d^2(2-\delta)/2)$.
\item For every $u \in \partial^1 \hat G_{r-1}$, 
\[ \sum_{\substack{v\in \partial^2 \hat G_{r-1}, v \sim u \\ \sigma(u) = \sigma(v)}} X_v = N_{aa}(u) \]
\[ \sum_{\substack{v\in \partial^2 \hat G_{r-1}, v \sim u \\ \sigma(u) \ne \sigma(v)}} X_v = N_{bb}(u) \]

\item $A_r, B_r, D_r$ hold.
\end{enumerate}
Then  $( \hat G_{r}, \sigma_{r}, \tau_{r}) = (T_{r}, \hat \sigma_{r}, \hat \tau_{r})$.
\end{lemma}  

\noindent
Next we define $C_r = \{| \partial \hat G_s | \le 4^s d^{2s} \min \{ \log n_1, \log (n_2/n_1)  \} , \forall s \le r+1   \} $.Then,
\begin{lemma}
\label{lem:4445}
Whp, $A_r, B_r, C_r$, and  $D_r$ hold for all $1 \le r \le R$ and \\ $|\hat G_R| = O( \min \{ n_1^{1/8}, (n_2/n_1)^{1/8} \} ) $.  
\end{lemma}

\begin{proof}
As in Lemma 4.4 from \cite{mossel2012stochastic}, stochastic domination and a Chernoff bound show that $A_r, B_r, C_r$ hold whp: the distribution of $N^2(u_i)$ is dominated by a $\text{Bin}(n_1, 4 d^2/n_1)$.  If $v \in V_2$ is revealed to be a neighbor of $u_i$, then the probability it has at least $2$ additional neighbors in $V_1$ is bounded by $O(n_1^2 p^2) = O( n_1/n_2)$.  Given that $| \hat G_R| = O( \min \{ n_1^{1/8}, (n_2/n_1)^{1/8} \} ) $, a union bound gives that $D_r$ holds for all $1 \le r \le R$ whp.  
\end{proof}

\noindent
Finally, we  complete the proof of Lemma \ref{lem:Rdepth}.
Condition on the event that $\beta_1 = O(n_1^{-1/3})$ and $\beta_2 = O(n_2^{-1/3})$, which occurs with probability $\ge 1- \exp(-\Theta(n_1^{1/3}))$.  Condition also on the event that the number of edges incident to all explored vertices in $V_1$ is at most $2 n_2 p \min \{ n_1^{1/8}, (n_2/n_1)^{1/8} \}$.  Under these two events we have $|V_1^+(i)|, |V_1^-(i)| = n_1/2 (1+ O(n_1^{-1/3}))$ and $|V_2^+(i)|, |V_2^-(i)| = n_2/2(1+O(n_2^{-1/3}))$ for all $1 \le i \le R$.  
\\

\noindent
For $u_i \in V_1$, the distribution of the number of its unexplored neighbors of degree $2$ and label $+1$ is $\text{Bin}(|V_2^+(i)|, p^{\sigma(u_i),+}  )$, and the distribution of the number of its unexplored neighbors of degree $2$ and label $-1$ is $\text{Bin}(|V_2^-(i)|, p^{\sigma(u_i),-}  )$
where 
\begin{align*}
p^{ +,+} &= |V_1^+(i)| \del^2 p^2 (1-\del p)^{|V_1^+(i)|-1} (1 -(2-\del)p)^{|V_1^-(i)|} \\
& + |V_1^-(i)| \del (2-\del) p^2 (1-\del p)^{|V_1^+(i)|} (1 -(2-\del)p)^{|V_1^-(i)|-1} \\
&= n_1 \del p^2 (1+ O(n_1^{-1/3})) 
\end{align*}
\noindent
and likewise 
\begin{align*}
p^{-,-} &= n_1 \del p^2 (1+ O(n_1^{-1/3})) \\
p^{+,-}, p^{-,+} &= n_1 (2-\del)p^2 (1+ O(n_1^{-1/3})).
\end{align*}

\noindent
From Lemma 4.6 in \cite{mossel2012stochastic}, we then have that for $\sigma(u) \in \{ \pm 1 \}$,
\begin{align*}
&\|  \text{Bin}(|V_2^{\sigma(u)}(i)|, p^{\sigma(u),\sigma(u)})  - \pois (d^2 \del /2)  \|_{TV} = O( n_1^{-1/3})  \\
 &\|  \text{Bin}(|V_2^{-\sigma(u)}(i)|, p^{\sigma(u),-\sigma(u)})  - \pois (d^2 (2-\del) /2)  \|_{TV} = O( n_1^{-1/3}). 
\end{align*}

\noindent
Since we have $| \hat G_R| \le n_1^{1/8}$ whp,  a union bound over $r = 1,\ldots, R$ shows that there is a coupling so that whp for all $r \le R$ and every $u \in \partial^1 \hat G_{r-1}$, $N_{aa}(u) + N_{ab}(u)= Y_u^a$; $N_{ba}(u) + N_{bb}(u)= Y_u^b$.
\\

\noindent
Next, we show that the probability the second neighbor of a vertex of degree $2$ in $V_2$ has the same label is close to $\del/2$.  Let $u_i$ be the current active vertex, $v$ a neighbor of $u_i$ of degree $2$, and $w$ the second unexplored neighbor of $v$.  Then
\begin{align*}
\Pr [ \sigma(w) = \tau(v) | u_i \sim v, d(v) =2 ] & =   \del/2 + O(n_1^{-1/3}).
\end{align*}

This shows that the coupling can be extended to the $X_v$'s, and that whp under this coupling $(\hat G_{R}, \sigma_{R}, \tau_{R}) = (T_{R}, \hat \sigma_{R}, \hat \tau_{R}) $.

\end{proof}

Let $V_2^{(0)}$ and $V_2^{(1)}$ be the subsets of $V_2$ of degree $0$ and $1$ respectively. Let $V_2^{(\ge 2)}$ be the vertices of $V_2$ of degree at least $2$.  Recall that  $\hat G$ is the subgraph of $G$ induced by $V_1$ and $V _2^{(\ge 2)}$.  From $\hat G$ we can determine the set $V_2^{(\le 1)} = V_2^{(0)} \cup V_2^{(1)}$ but not the two sets individually.

The following is an analogue of Lemma 4.7 in \cite{mossel2012stochastic}.  It says that conditioned on the labels at depth $R$ in $\hat G$ from the root $\rho$, neither the graph outside the $R$-neighborhood, nor the vertices in $V_2^{(\le 1)}$ contain significant information about the label of $\rho$. 
\begin{lemma}
\label{lem:47}
Let $A, B, C $ be a partition of $ V_1 \cup V_2^{(\ge 2)}$  so that $B$ separates $A$ and $C$ in $\hat G$. Assume $| A \cup B | = o(\sqrt{n_1}) $. Then 
\[ \Pr [ \sigma_A, \tau_A |  \sigma_{B \cup C}, \tau _{B \cup C}, G]  =(1+o(1)) \Pr[\sigma_A, \tau_A |  \sigma_{B}, \tau _{B}, \hat G  ]   \]
whp over $G, \sigma$ and $\tau$. 
\end{lemma}
\noindent
We delay the proof of Lemma~\ref{lem:47} to the Appendix.
\\

\noindent
Now, we can finish the proof of Theorem~\ref{thm:noDetect}. By the monotonicity of conditional variance,
\[ \var(\sigma(\rho) |G, \sigma(w), \sigma_{\partial^1 \hat G_R}, \tau_{\partial^2 \hat G_R}) \le \var(\sigma(\rho) | G, \sigma(w)).\]
 Then whp $w \notin \hat G_R$, and so by Lemma \ref{lem:47}, 
\[\var(\sigma(\rho) |G, \sigma(w), \sigma_{\partial^1 \hat G_R}, \tau_{\partial^2 \hat G_R}) \to \var(\sigma(\rho) | G, \sigma_{\partial^1 \hat G_R}, \tau_{\partial^2 \hat G_R}) \] (since $\sigma(\rho)$ and $\sigma(w)$ are independent given $G, \sigma_{\partial^1 \hat G_R}, \tau_{\partial^2 \hat G_R}$).
By Lemma \ref{lem:Rdepth}, 
\[ \left | \var(\sigma(\rho) | G, \sigma_{\partial^1 \hat G_R}, \tau_{\partial^2 \hat G_R}) - \var( \hat \sigma(\rho) | T, \hat \sigma_{\partial^1 T_R}, \hat \tau_{\partial^2 T_R})  \right | \to 0. \]
From the results of \cite{evans2000broadcasting} and the condition $p \le \frac{1}{(\del-1)^2 \sqrt{n_1 n_2} }$,  
\[ \var( \hat \sigma(\rho) | T, \hat \sigma_{\partial^1 T_R}, \hat \tau_{\partial^2 T_R}) = 1-o(1) .\]
Thus, whp 
 \[ \var(\sigma(\rho) | G, \sigma_{\partial^1 \hat G_R}, \tau_{\partial^2 \hat G_R})  =1 -o(1) \]
  as well.  
This implies that the labels of $\rho$ and $w$ are asymptotically independent and in particular proves Theorem~\ref{thm:noDetect}.

\section{Proof of Theorem \ref{thm:main}: Recovery}

We will follow a similar framework to prove both parts of Theorem \ref{thm:main}. Recalling $M$ to be the adjacency matrix, let $B = M M^T - \text{diag}(M M^T)$ and  $D_V = \text{diag}(M M^T)$.


A simple computation shows that the second eigenvector of $\E B$ is the vector $\sigma$ that we wish to recover; we will consider the different perturbations of $\E B$ that arise with the three spectral algorithms and show that at the respective thresholds, the second eigenvector of the resulting matrix is close to $\sigma$. To analyze the diagonal deletion SVD, we  must show that the second eigenvector of $B$ is highly correlated with $\sigma$ (the addition of a constant multiple of the identity matrix does not change the eigenvectors). The main step is to bound the spectral norm $\| B - \E B \|$.  Since the entries of $B$ are not independent, we will decompose $B$ into a sequence of matrices based on subgraphs induced by vertices of a given degree in $V_2$.  This (Lemma \ref{lem:norms}) is the most technical part of the work. 

To analyze the unmodified SVD, we write $M M^T = \E B + (B - \E B)  + \E D_V + (D_V - \E D_V)$.
The left singular vectors of $M$ are the eigenvectors of $MM^T$. $\E B$ has $\sigma $ as its second eigenvector and $\E D_V$ is a multiples of the identity matrix and so adding it does not change the eigenvectors. As above we bound $\| B - \E B \|$ and what remains is showing that the difference of the matrix $D_V$ with its expectation has small spectral norms at the respective thresholds; this involves simple bounds on the fluctuations of independent random variables.

We will assume that $\sigma$ and $\tau$  assign $+1$ and $-1$ labels to an equal number of vertices; this allows for a clearer presentation, but is not necessary to the argument.  We will treat $\sigma$ and $\tau$ as unknown but fixed, and so expectations and probabilities will all be conditioned on the labelings.

The main technical lemma is the following:

\begin{lemma}
\label{lem:norms}
Define $B, D_V$ as above. Assume $n_1, n_2,$ and $p$ are as in Theorem~\ref{thm:main}.  Then there exists an absolute constant $C$ so that
\begin{enumerate}
\item $\E B = \lam_1 J/n_1 + \lam_2 \sigma \sigma^T /n_1$, with $\lam_1 = n_1 n_2 p^2$ and $\lam_2 = (\del-1)^2 n_1 n_2 p^2$, where $J$ is the all ones $n_1 \times n_1 $ matrix.  
\item For $p \ge n_1^{-1/2} n_2 ^{-1/2} \log n_1$, $\| B - \E B \| \le C n_1^{1/2} n_2^{1/2} p$ whp.
\item $\E D_V$ is a multiple of the identity matrix.
\item For $p \ge n_1^{-2/3} n_2^{-1/3} \log n_1$, $\| D_V - \E D_V \| \le C \sqrt{ n_2 p \log n_1}$ whp.
\end{enumerate}

\end{lemma}

This is proved in Appendix \ref{sec:proveBBound}.

 We also will use the following lemma from \cite{lelarge2013reconstruction} to round a unit vector with high correlation with $\sigma$ to a $\pm 1$ vector that denotes a partition: 
\begin{lemma}[\cite{lelarge2013reconstruction}]
\label{lem:lelarge2}
For any $x \in \{-1,+1\}^n$ and $y\in \mathbb{R}^n$ with $\|y\|=1$ we have 
$$d(x,\text{sign}(y)) \le n \left \| \frac{x}{\sqrt{n}} - y \right \| ^2,$$
where $d$ represents the Hamming distance.
\end{lemma}

The next lemma is a classic eigenvector perturbation theorem.  Denote by $P_{A}(S)$  the orthogonal projection onto the subspace spanned by the eigenvectors of $A$ corresponding to those of its eigenvalues that lie in $S$.

\begin{lemma}[\cite{davis1970rotation}]
\label{lem:daviskahan}

Let $A$ be an $n \times n$ symmetric matrix with $|\lam_1| \ge | \lam_2 | \ge \dots $, with $|\lam _k| - |\lam _{k+1}| \ge 2 \del$. Let $B$ be a symmetric matrix with $\| B \| < \del$. Let $A_k$ and $(A+B)_k$ be the spaces spanned by the top $k$ eigenvectors of the respective matrices. Then
\[ \sin ( A_k, (A+B)_k)= \| P_{A_k} - P_{(A+B)_k} \| \le \frac{ \| B \|}{ \del }    \]
In particular, If $|\lam_1| - |\lam_2| \ge 2 \del$, $|\lam_2| - |\lam_3| \ge 2 \del$, $\| B \| < \del$, and $e_2(A)$, $e_2(A+B)$ are the second (unit) eigenvectors of $A$ and $A+B$, respectively, satisfying $e_2(A) \cdot e_2(A+B)\ge 0$, then $\| e_2(A) - e_2(A+B) \| \le \frac{4 \| B \| }{\del}$. 
\end{lemma}

\begin{proof}  
In the particular case, let $u_1, u_2$ be the first two eigenvectors of $A$, and $v_1, v_2$ the first two eigenvectors of $A+B$, with signs chosen so that $u_1 \cdot v_1, u_2 \cdot v_2 \ge 0$. Let $\eps =\|B\|/\del$.  First we apply the lemma with $k=1$ to get $\| P_{u_1} - P_{v_1}\| \le \eps$.  Applying this to $v_1$, we get $\| P_{u_1} v_1 - v_1 \| \le \eps$, and so $\| P_{u_1} v_1 \| \ge 1- \eps$. The triangle inequality gives $\| u_1 - v_1 \| \le 2 \eps$.   Now apply the lemma with $k=2$ to get $\| P_{A_2} - P_{(A+B)_2} \|\le \eps$.  Apply this to $v_2$ and use the triangle inequality again to get $\| u_2 - v_2 \| \le 4 \eps$. 
\end{proof}

Now using Lemmas \ref{lem:norms}, \ref{lem:lelarge2}, and \ref{lem:daviskahan}  we prove parts 1 and 2 of Theorem \ref{thm:main}. 
\\

\paragraph{\textbf{Diagonal deletion SVD}}
Let $p \ge n_1^{-1/2} n_2 ^{-1/2} \log n_1$. Part 1 of Lemma \ref{lem:norms} shows that if we had access to the second eigenvector of $\E B$, we would recover $\sigma$ exactly. (The addition of a multiple of the identity matrix does not change the eigenvectors).   Instead we have access to $B = \E B + (B- \E B)$, a noisy version of the matrix we want.  We use a matrix perturbation inequality to show that the top eigenvectors of the noisy version are not too far from the original eigenvectors. 

Let $y_1$ and $y_2$ be the top two eigenvectors of $B$, and  $\hat B$ be the space spanned by $y_1$ and $y_2$, and $(\mathbb{E}B)_2$ the space spanned by the top two eigenvectors of $\mathbb{E}B$.  Then Lemma \ref{lem:daviskahan} gives
\begin{align*}
 \sin( (\E B)_2, \hat B) &\le \frac{C \| B- \E B \| }{\lam_2}  \le C \frac{n_1^{1/2} n_2^{1/2} p}{(\del-1)^2 n_1 n_2 p^2}  =O\left ( \frac{1}{\log n_1}  \right) 
\end{align*}
where the inequality holds whp by Lemma \ref{lem:norms}.  Assuming $\del \in (0,2)$, we use the particular case of Lemma \ref{lem:daviskahan}  to show that $\| y_2 - \sigma /\sqrt{n_1} \| =O(\log^{-1} n_1)$.  We round $y_2$ by signs to get $z$, and then apply Lemma \ref{lem:lelarge2} to show that whp the algorithm recovers $1-o(1)$ fraction of the coordinates of $\sigma$.  (If $\del =0$ or $2$, then instead of taking the second eigenvector, we take the component of $\hat B$ perpendicular to the all ones vector and get the same result). 
\\


\paragraph{\textbf{The SVD}}

Let $p \ge n_1^{-2/3} n_2 ^{-1/3} \log n_1$. Let $y_1$ and $y_2$ be the top two left singular vectors of $ M$, and  $ M_2$ be the space spanned by $y_1$ and $y_2$.  $y_1$ and $y_2$ are  the top two eigenvectors of $ M  M^T = B + D_V $.   Again  Lemma \ref{lem:daviskahan} gives that whp,
\begin{align*}
 \sin( (\E B)_2, M_2) &\le C \frac{ \| B- \E B \| + \|D_V - \E D_V \| }{\lam_2}  \le \frac{C _1 n_1^{1/2} n_2^{1/2} p + C_2 \sqrt{ n_2 p \log n_1}}{(\del-1)^2 n_1 n_2 p^2} =O\left ( \frac{1}{\log n_1}  \right). 
\end{align*}
This gives $\| y_2 - \sigma /\sqrt{n_1} \| =O(\log^{-1} n_1)$, and shows that the SVD algorithm recovers $\sigma $ whp.  Note that in this case $\| D_V - \E D_V \| \gg \| B- \E B \|$.  It is these fluctuations on the diagonal that explain the poor performance of the SVD and its need for a higher edge density for success.

\section{Proof of Theorem~\ref{thm:SVDfail}: Failure of the vanilla SVD}

Here we again use a matrix perturbation lemma, but in the opposite way: we will show that the `noise matrix' $( D_V - \E D_V)$ has a large spectral norm (and an eigenvalue gap), and thus adding the `signal matrix' approximately preserves the space spanned by the top eigenvalues. This shows that the top $t$ eigenvectors of $B+ D_V$ have almost all their weight on a small number of coordinates and is enough to conclude that they cannot be close to the planted vector $\sigma$. 

The perturbation lemma we use is a generalization of the Davis-Kahan theorem found in \cite{bhatia1997matrix}.

\begin{lemma}[\cite{bhatia1997matrix}]
\label{lem:perturb}
Let $A$ and $B$ be $n \times n$ symmetric matrices with the eigenvalues of $A$ ordered $\lam_1 \ge \lam _2 \ge \dots  \lam_n$.   Suppose $r>k$,   $ \lam _k  -  \lam _r > 2\del$, and $\|B \| \le \del$.  Let $A_r$ denote the subspace spanned by the first $r$ eigenvectors of $A$ and likewise for $(A+B)_k$.  Then
\[ \| P_{A_r^\perp} P_{(A+B)_k} \| \le \frac {\|B \|}{\del}.   \]
In particular, if $v_k$ is the $k^{th}$ unit eigenvector of $(A+B)$, then there is some unit vector $u \in A_r$ so that 
\[ \| u - v_k \| \le \frac{4 \| B \|}{\del}  .\]
\end{lemma}

\begin{proof}This lemma is a special case of Theorem VII.3.1 from \cite{bhatia1997matrix}, itself a generalization of the Davis-Kahan theorem. 
In the particular case, write $v_k =  v_k^{(r)} + \overline v_k^{(r)}$ where $v_k^{(r)} \in A_r$ and $\overline v_k^{(r)} \in A_r^\perp$. Let $\eps = \|B \| / \del$. Then, by multiplying we get
\[  P_{A_r^\perp} P_{(A+B)_k} v_k = P_{A_r^\perp} v_k = \overline v_k^{(r)}. \]
We see that $\| \overline v_k^{(r)} \| \le \eps$, and thus $\|  v_k^{(r)} \| \ge 1- \eps$. Take $u = v_k^{(r)}/\| v_k^{(r)} \|$ and use the triangle inequality to complete the lemma: $u \in A_r$ and $\| u - v_k \| \le 4 \eps$. 
\end{proof}

We also need to analyze the degrees of the vertices in $V_1$.  The following lemma gives some basic information about the degree sequence:

\begin{lemma}
\label{lem:degrees}
Let $d_1, \dots d_{n_1}$ be the sequence of degrees of vertices in $V_1$.  Then there exist constants $c_1, c_2, c_3 $ so  that
\begin{enumerate}
\item The $d_i$'s are independent and identically distributed, with distribution $d_i \sim \operatorname{Bin}(n_2/2, \del p) +  \operatorname{Bin}(n_2/2,(2- \del) p)$. 
\item $\E d_i = n_2 p$.
\item Whp, $\displaystyle \max_{i} d_i \le n_2 p + c_1 \sqrt{ n_2 p \log n_1}  $.
\item  Whp, $\left | \{ i: d_i \ge  n_2 p + c_2 \sqrt{ n_2 p \log n_1}  \}  \right | \ge n_1^{1/3}$.
\item Whp, $\left | \{ i: d_i \ge  n_2 p + c_3 \sqrt{ n_2 p \log \log n_1}  \}  \right | \le n_1 / \log n_1$.
\end{enumerate}

\end{lemma}

The lemma follows from basic Chernoff bounds and the first- and second-moment methods. Now, we can finish the proof of Theorem~\ref{thm:SVDfail}.

\begin{proof}[Proof of Theorem~\ref{thm:SVDfail}]
Let $p = c n_1^{-2/3} n_2^{-1/3}$.  The left singular vectors of $M$ are the eigenvectors of $B+ D_V$.  Recall that $D_V$ is a diagonal matrix with the $i$th  entry the degree of the $i$th vertex of $V_1$.  $\E D_V$ is therefore a multiple of the identity matrix, and so subtracting $\E D_V$ from $B+ D_V$ does not change its eigenvectors. The standard basis vectors form an orthonormal set of eigenvectors of $ D_V -\E D_V$.  

For the constants $c_2, c_3$ in  Lemma \ref{lem:degrees}, let $\eta_1 = c_2 \sqrt{n_2 p \log n_1}$ and $\eta_2 = c_3 \sqrt{n_2 p \log \log n_1}$. Order the eigenvalues of $D_V - \E D_V$ as $\lam_1 \ge \lam_2 \ge \dots \ge \lam _n$ and let $r$ be the smallest integer  such that $\lam_r < \eta_2$.  Then we have $\lam_i - \lam_r \ge c \sqrt{n_2 p \log n_1}$ for all $1 \le i \le t$.  From Lemma \ref{lem:degrees}, $r \le n_1 /\log n_1$.

We now bound 
\begin{align*}
\| B\| &\le \| \E B\| + \| B - \E B\| \le n_1 n_2 p^2 +  C n_1^{1/2} n_2^{1/2} p.
\end{align*}
Now Lemma \ref{lem:perturb} says that if $v_i$ is the $i$th eigenvector of $ D_V - \E D_V + B$, then there is a vector $u$ in the span of the first $r$ eigenvectors of $D_V - \E D_V$ so that 
\begin{align*}
\| v_i - u \| & \le C \frac{  n_1 n_2 p^2 +   n_1^{1/2} n_2^{1/2} p  }{  \sqrt{n_2 p \log n_1}   } = O \left( \frac{1}{\sqrt{\log n_1}} \right ).
\end{align*}

The span of the first $r$ eigenvectors of $D_V - \E D_V$ is supported on only $r$ coordinates, so $u$ is far from $\overline \sigma = \sigma / \sqrt {n_1}$:  
\[ \| u - \overline \sigma  \| \ge \sqrt{ 2 - 2 \sqrt{r/n_1}} = \sqrt 2 - O(1/\sqrt{\log n_1}). \]
 By the triangle inequality, $v_i$ must also be far from $\overline \sigma$: $| v_i \cdot \overline \sigma| = O(1/\sqrt{\log n_1})$.  This proves Theorem \ref{thm:SVDfail}.
\end{proof}

\section*{Acknowledgements}
We thank the Institute for Mathematics and its Applications (IMA) in Minneapolis, where part of this work was done, for its support and hospitality.  


\bibliography{planted2}
\bibliographystyle{plain}

\appendix

\section{Proof of Lemma~\ref{lem:47}.}
\begin{proof}

\begin{align*}
\Pr [ \sigma_A, \tau_A |  \sigma_{B \cup C}, \tau _{B \cup C}, G] &= \frac{ \Pr [ \sigma_A, \tau_A, G |  \sigma_{B \cup C}, \tau _{B \cup C}, \hat G] }{\Pr [ G | \sigma_{B \cup C}, \tau _{B \cup C}, \hat G ]    } \\
&= \Pr [ \sigma_A, \tau_A |  \sigma_{B \cup C}, \tau _{B \cup C}, \hat G] \cdot \frac{\Pr [ G | \sigma_A, \tau_A, \sigma_{B \cup C}, \tau _{B \cup C}, \hat G ]   }  { \Pr [ G | \sigma_{B \cup C}, \tau _{B \cup C}, \hat G ]    } \\
&= \Pr [ \sigma_A, \tau_A |  \sigma_{B \cup C}, \tau _{B \cup C}, \hat G] \cdot \frac{\Pr [ G | \sigma,  \hat G ]   }  { \Pr [ G | \sigma_{B \cup C},  \hat G ]    }
\end{align*}

We now show that the last factor is $1+o(1)$ whp over $\sigma, \tau$, and $G$.  

\begin{lemma}
\label{lem:degree01}
Let $U \subseteq V_1$, and $\sigma_U$ the restriction of $\sigma$ to $U$.    
Then
\begin{align*}
\Pr[ G | \sigma, \hat G] = (1+o(1)) \Pr[ G | \sigma_U, \hat G]
\end{align*}
whp over the choices of $\sigma,  G$. 
\end{lemma}

We leave the proof of Lemma~\ref{lem:degree01} to Appendix B.


To prove Lemma \ref{lem:47}, it remains to show

\begin{equation}
\label{47eq}
 \Pr [ \sigma_A, \tau_A |  \sigma_{B \cup C}, \tau _{B \cup C}, \hat G] = (1+o(1)) \Pr[\sigma_A, \tau_A |  \sigma_{B}, \tau _{B}, \hat G  ]  
 \end{equation}
whp over $\sigma, \tau, G$.  Now that we have removed vertices of degree $0$ and $1$ from $V_2$, the proof proceeds along the same lines as the proof of Lemma 4.7 of \cite{mossel2012stochastic}. 

For $u\in V_1, v \in V_2$, define 
\[ \psi_{u,v}(\hat G, \sigma, \tau) = 
\begin{cases}
 \del p \, , \text{if } (u,v) \in E(\hat G), \sigma (u) = \tau (v) \\
 (2-\del)p \, , \text{if } (u,v) \in E(\hat G), \sigma (u) \ne \tau (v) \\
 1- \del p \, , \text{if } (u,v) \notin E(\hat G), \sigma (u) = \tau (v)\\
  1 - (2-\del)p \, , \text{if }(u,v) \notin E(\hat G), \sigma (u) \ne \tau (v). 
\end{cases}
   \]
   Define $Q_{U_1, U_2}$ to be the product of $\psi_{u,v}(\hat G, \sigma, \tau)$ over all $u \in U_1, v \in U_2$. Define $$Q^-(\hat G, \sigma, \tau) = \Pr[ \wedge_{e \in V_1 \times V_2^{(\le 1)}} e \notin G | \sigma, \tau_{V_2^{(\ge 2)}}].$$ 

We denote by $\eta$ and $\phi$ labelings of $V_1$ and $V_2$ respectively.  $\eta_A$ refers to the restriction of $\eta $ to the set $A \cap V_1$, and so on. We write $\sigma_{AB}$ instead of $\sigma_{A \cup B}$ for cleaner notation.   Equation (\ref{47eq}) is equivalent to 
\begin{equation}
\label{eq:47.2}
\frac{\Pr[ \hat G | \sigma_{A  B  C}, \tau_{A  B  C}] }{\Pr[ \hat G | \sigma_{B  C}, \tau_{B  C}]  } = (1+o(1)) \cdot \frac{\Pr[ \hat G | \sigma_{A  B }, \tau_{A  B }] }{\Pr[ \hat G | \sigma_{B }, \tau_{B }]  }
\end{equation}

We rewrite the LHS of (\ref{eq:47.2}) as 
\begin{align}
\nonumber
&\frac{ Q_{A,A B}(\sigma_{A  B  C}, \tau_{A  B  C}) Q_{A,C}(\sigma_{A  B  C}, \tau_{A  B  C} ) Q_{B  C, B  C}(\sigma_{A  B  C}, \tau_{A  B  C}) Q^-(\hat G, \sigma_{A  B  C}, \tau_{A  B  C})  }{ \sum_{\substack{\eta: \eta_{B  C}=\sigma_{B  C} \\ \phi: \phi_{B  C}= \tau_{B  C}} }  Q_{A,A B}(\eta_{A  B  C}, \phi_{A  B  C}) Q_{A,C}(\eta_{A  B  C}, \phi_{A  B  C} ) Q_{B  C, B  C}(\eta_{A  B  C}, \phi_{A  B  C}) Q^-(\hat G, \eta_{A  B  C}, \phi_{A  B  C}) }  \\
\nonumber
&=  \frac{ Q_{A,A B}(\sigma_{A  B }, \tau_{A  B }) Q_{A,C}(\sigma_{A   C}, \tau_{A   C} ) Q_{B  C, B  C}(\sigma_{  B  C}, \tau_{  B  C}) Q^-(\hat G, \sigma_{A  B  C}, \tau_{A  B  C})  }{ \sum_{\substack{\eta: \eta_{B  C}=\sigma_{B  C} \\ \phi: \phi_{B  C}= \tau_{B  C}} }  Q_{A,A B}(\eta_{A  B  }, \phi_{A  B  }) Q_{A,C}(\eta_{A    C}, \phi_{A    C} ) Q_{B  C, B  C}(\sigma_{  B  C}, \tau_{  B  C}) Q^-(\hat G, \eta_{A  B  C}, \phi_{A  B  C}) }  \\
&=  \frac{ Q_{A,A B}(\sigma_{A  B }, \tau_{A  B }) Q_{A,C}(\sigma_{A   C}, \tau_{A   C} ) Q^-(\hat G, \sigma_{A  B  C}, \tau_{A  B  C})  }{ \sum_{\substack{\eta: \eta_{B  C}=\sigma_{B  C} \\ \phi: \phi_{B  C}= \tau_{B  C}} }  Q_{A,A B}(\eta_{A  B  }, \phi_{A  B  }) Q_{A,C}(\eta_{A    C}, \phi_{A    C} )  Q^-(\hat G, \eta_{A  B  C}, \phi_{A  B  C}) } 
\label{eq47.3} 
\end{align}

This is a similar expression as encountered in the proof of Lemma 4.7 in \cite{mossel2012stochastic} apart from the factors involving $Q^-$.  To address these factors we use the following Lemma:

\begin{lemma}
\label{lem:qm}
Let $U \subseteq V_1 \cup V_2^{(\ge 2)}$ so that $|V_1 \cup V_2^{(\ge 2)} \setminus U| =o(n_1^{1/2})$. Then for any $\eta, \phi$ so that $\eta_U = \sigma_U$, $\phi_U = \tau_U$,
\[Q^-(\hat G, \sigma, \tau) = (1+o(1))  Q^-(\hat G, \eta, \phi) \]
with probability $1-o(1)$ over the choice of $\sigma, \tau, \hat G$. 
\end{lemma}

The proof is similar to that of Lemma \ref{lem:degree01}, where we use the fact that $U^c$ is small to show that $\beta_1$ is essentially determined on $U$.

With Lemma \ref{lem:qm}, equation (\ref{eq47.3}) becomes
\begin{align*}
&= (1+o(1)) \cdot \frac{ Q_{A,A B}(\sigma_{A  B }, \tau_{A  B }) Q_{A,C}(\sigma_{A   C}, \tau_{A   C} )   Q^-(\hat G, \sigma_{A  B  C}, \tau_{A  B  C})  }{ \sum_{\substack{\eta: \eta_{B  C}=\sigma_{B  C} \\ \phi: \phi_{B  C}= \tau_{B  C}} }  Q_{A,A B}(\eta_{A  B  }, \phi_{A  B  }) Q_{A,C}(\eta_{A    C}, \phi_{A    C} )  Q^-(\hat G, \sigma_{A  B  C}, \tau_{A  B  C}) } \\
&=(1+o(1)) \cdot \frac{ Q_{A,A B}(\sigma_{A  B }, \tau_{A  B }) Q_{A,C}(\sigma_{A   C}, \tau_{A   C} )   }{ \sum_{\substack{\eta: \eta_{B  C}=\sigma_{B  C} \\ \phi: \phi_{B  C}= \tau_{B  C}} }  Q_{A,A B}(\eta_{A  B  }, \phi_{A  B  }) Q_{A,C}(\eta_{A    C}, \phi_{A    C} )  } 
\end{align*}
with probability $1-o(1)$.  Now using (2) from \cite{mossel2012stochastic} gives
\begin{align*}
&= (1+o(1)) \cdot \frac{ Q_{A,A B}(\sigma_{A  B }, \tau_{A  B })    }{ \sum_{\substack{\eta: \eta_{B  C}=\sigma_{B  C} \\ \phi: \phi_{B  C}= \tau_{B  C}} }  Q_{A,A B}(\eta_{A  B  }, \phi_{A  B  })   }
\end{align*}

with probability $1-o(1)$. Now since $Q_{A, AB}$ does not depend on $\sigma_C, \tau_C$, we have
\begin{equation}
\label{eq:47.4}
= (1+o(1)) \cdot \frac{ Q_{A,A B}(\sigma_{A  B }, \tau_{A  B })    }{ \sum_{\substack{\eta: \eta_{B  }=\sigma_{B  } \\ \phi: \phi_{B  }= \tau_{B  }} }  Q_{A,A B}(\eta_{A  B  }, \phi_{A  B  })   }
\end{equation}

Now we can proceed similarly with the RHS of (\ref{eq:47.2}):

\begin{align}
\nonumber
&\frac{ \sum_{\substack{\eta: \eta_{AB  }=\sigma_{AB  } \\ \phi: \phi_{AB  }= \tau_{AB  }} } Q_{A,A B}(\sigma_{A  B  }, \tau_{A  B  }) Q_{A,C}(\eta_{A    C}, \phi_{A    C} ) Q_{B  C, B  C}(\eta_{  B  C}, \phi_{  B  C}) Q^-(\hat G, \eta_{A  B  C}, \phi_{A  B  C})  }{ \sum_{\substack{\eta: \eta_{B  }=\sigma_{B  } \\ \phi: \phi_{B  }= \tau_{B  }} }  Q_{A,A B}(\eta_{A  B  }, \phi_{A  B  }) Q_{A,C}(\eta_{A    C}, \phi_{A    C} ) Q_{B  C, B  C}(\eta_{  B  C}, \phi_{  B  C}) Q^-(\hat G, \eta_{A  B  C}, \phi_{A  B  C}) }  \\
\nonumber
&= \frac{  Q_{A,A B}(\sigma_{A  B  }, \tau_{A  B  })  \sum_{\substack{\eta: \eta_{AB  }=\sigma_{AB  } \\ \phi: \phi_{AB  }= \tau_{AB  }} }Q_{A,C}(\eta_{A    C}, \phi_{A    C} ) Q_{B  C, B  C}(\eta_{  B  C}, \phi_{  B  C}) Q^-(\hat G, \eta_{A  B  C}, \phi_{A  B  C})  }{ \sum_{\substack{\eta: \eta_{B  }=\sigma_{B  } \\ \phi: \phi_{B  }= \tau_{B  }} }  Q_{A,A B}(\eta_{A  B  }, \phi_{A  B  }) Q_{A,C}(\eta_{A    C}, \phi_{A    C} ) Q_{B  C, B  C}(\eta_{  B  C}, \phi_{  B  C}) Q^-(\hat G, \eta_{A  B  C}, \phi_{A  B  C}) } \\
\nonumber
&= (1+o(1)) \cdot \frac{  Q_{A,A B}(\sigma_{A  B  }, \tau_{A  B  })  \sum_{\substack{\eta: \eta_{AB  }=\sigma_{AB  } \\ \phi: \phi_{AB  }= \tau_{AB  }} } Q_{B  C, B  C}(\eta_{  B  C}, \phi_{  B  C}) Q^-(\hat G, \eta_{A  B  C}, \phi_{A  B  C})  }{ \sum_{\substack{\eta: \eta_{B  }=\sigma_{B  } \\ \phi: \phi_{B  }= \tau_{B  }} }  Q_{A,A B}(\eta_{A  B  }, \phi_{A  B  })  Q_{B  C, B  C}(\eta_{  B  C}, \phi_{  B  C}) Q^-(\hat G, \eta_{A  B  C}, \phi_{A  B  C}) } \\
\nonumber
&= (1+o(1)) \cdot \frac{  Q_{A,A B}(\sigma_{A  B  }, \tau_{A  B  })  \sum_{\substack{\eta: \eta_{AB  }=\sigma_{AB  } \\ \phi: \phi_{AB  }= \tau_{AB  }} } Q_{B  C, B  C}(\eta_{  B  C}, \phi_{  B  C}) Q^-(\eta_{A  B  C}, \phi_{A  B  C})  }{ \sum_{\substack{\eta: \eta_{B  }=\sigma_{B  } \\ \phi: \phi_{B  }= \tau_{B  }} }  Q_{A,A B}(\eta_{A  B  }, \phi_{A  B  })  Q_{B  C, B  C}(\eta_{  B  C}, \phi_{  B  C}) Q^-(\sigma_A,\eta_{  B  C}, \tau_A,\phi_{  B  C}) }  \\
\nonumber
&=  (1+o(1)) \cdot \frac{  Q_{A,A B}(\sigma_{A  B  }, \tau_{A  B  })  \sum_{\substack{\eta: \eta_{AB  }=\sigma_{AB  } \\ \phi: \phi_{AB  }= \tau_{AB  }} } Q_{B  C, B  C}(\eta_{  B  C}, \phi_{  B  C}) Q^-(\hat G, \eta_{A  B  C}, \phi_{A  B  C})  }{ \sum_{\substack{\eta: \eta_{B  }=\sigma_{B  } \\ \phi: \phi_{B  }= \tau_{B  }} }  Q_{A,A B}(\eta_{A  B  }, \phi_{A  B  }) \cdot \sum_{\substack{\eta: \eta_{B  }=\sigma_{B  } \\ \phi: \phi_{B  }= \tau_{B  }} }  Q_{B  C, B  C}(\eta_{  B  C}, \phi_{  B  C}) Q^-(\hat G, \sigma_A,\eta_{  B  C}, \tau_A,\phi_{  B  C}) }  \\
\nonumber
&= (1+o(1)) \cdot \frac{  Q_{A,A B}(\sigma_{A  B  }, \tau_{A  B  })  \sum_{\substack{\eta: \eta_{AB  }=\sigma_{AB  } \\ \phi: \phi_{AB  }= \tau_{AB  }} } Q_{B  C, B  C}(\eta_{  B  C}, \phi_{  B  C}) Q^-(\hat G, \eta_{A  B  C}, \phi_{A  B  C})  }{ \sum_{\substack{\eta: \eta_{B  }=\sigma_{B  } \\ \phi: \phi_{B  }= \tau_{B  }} }  Q_{A,A B}(\eta_{A  B  }, \phi_{A  B  }) \cdot \sum_{\substack{\eta: \eta_{AB  }=\sigma_{AB  } \\ \phi: \phi_{AB  }= \tau_{AB  }} }  Q_{B  C, B  C}(\eta_{  B  C}, \phi_{  B  C}) Q^-(\hat G, \eta_{ A B  C}, \phi_{A  B  C}) }  \\
&= (1+o(1)) \cdot \frac{  Q_{A,A B}(\sigma_{A  B  }, \tau_{A  B  })    }{ \sum_{\substack{\eta: \eta_{B  }=\sigma_{B  } \\ \phi: \phi_{B  }= \tau_{B  }} }  Q_{A,A B}(\eta_{A  B  }, \phi_{A  B  }) }
\label{eq:47.6}
\end{align}
where we have again used Lemma \ref{lem:qm} and (2) from \cite{mossel2012stochastic}.  Now (\ref{eq:47.6}) matches (\ref{eq:47.4}) and so 
\[  \Pr [ \sigma_A, \tau_A |  \sigma_{B \cup C}, \tau _{B \cup C}, \hat G] = (1+o(1)) \Pr[\sigma_A, \tau_A |  \sigma_{B}, \tau _{B}, \hat G  ]  \]
with probability $1-o(1)$ over $\sigma, \tau, \hat G$, completing the proof of Lemma \ref{lem:47}.

\end{proof}

\section{Proof of Lemma~\ref{lem:degree01}.}
\begin{proof}

Note that conditioned on $\hat G$ and $\beta_1$, the distribution of $G$ is that of independently choosing degree $1$ or $0$ for each $v \in V_2$ of degree less than $ 2$, with probability that   depends only on $\beta_1$. 
Let $v \in V_2$. We can condition on $\beta_1$ and compute 
\begin{align*}
\Pr[d(v)=0 | \beta_1] &= \frac{1}{2} \Pr[d(v)=0 | \beta_1, \tau(v) = +1 ] + \frac{1}{2} \Pr[ d(v)=0 | \beta_1, \tau(v) =-1] \\
&= \frac{(1-\del p)^{n_1/2 }(1- (2-\del p))^{n_1/2}}{2} \left[ \left( \frac{1-\del p}{1 - (2-\del)p}\right) ^{\beta_1 n_1/2}+\left( \frac{1-(2-\del) p}{1 - \del p}\right) ^{\beta_1 n_1/2}   \right ] \\
&= (1-\del p)^{n_1/2 }(1- (2-\del p))^{n_1/2} (1+ \beta_1^2 n_1^2 (\del-1)^2 p^2/2 + O( \beta_1^4 n_1^4 p^4 )).
\end{align*}

Similarly,

\begin{align*}
&\Pr[d(v)=1 | \beta_1] = \frac{1}{2} \Pr[d(v)=1 | \beta_1, \tau(v) = +1 ] + \frac{1}{2} \Pr[ d(v)=1 | \beta_1, \tau(v) =-1] \\
&=\frac{1}{2} (1-\del p)^{n_1/2 }(1- (2-\del p))^{n_1/2} \left( \frac{1-\del p}{1 - (2-\del)p}\right) ^{\beta_1 n_1/2} \left( \frac{n_1(1+\beta_1)}{2} \frac{ \del p }{ 1- \del p  } + \frac{n_1(1-\beta_1)}{2}  \frac{(2-\del) p}{1 - (2-\del) p}   \right ) \\
&+\frac{1}{2} (1-\del p)^{n_1/2 }(1- (2-\del p))^{n_1/2} \left( \frac{1-\del p}{1 - (2-\del)p}\right) ^{-\beta_1 n_1/2} \left( \frac{n_1(1+\beta_1)}{2}  \frac{(2- \del) p }{ 1- (2-\del) p  } + \frac{n_1(1-\beta_1)}{2}  \frac{\del p}{1 - \del p}   \right ) \\
&= n_1 p/2 (1-\del p)^{n_1/2}  (1- (2-\del p))^{n_1/2} \left( \frac{1-\del p}{1 - (2-\del)p}\right) ^{\beta_1 n_1/2}\left(c(p) + \left(\frac{1-\del p}{1-(2-\del)p}\right)^{-\beta_1 n_1} d(p) \right)\\
&= \frac{n_1 p}{2} (1-\del p)^{n_1} \left[-c(p)\left(1+ \frac{\beta_1^2 n_1^2 (\del-1)^2 p^2}{2} + O( \beta_1^4 n_1^4 p^4 )\right) + (c(p) - d(p)) \left( \frac{1-\del p}{1 - (2-\del)p}\right) ^{-\beta_1 n_1/2} \right]
\end{align*}
where $c(p) = \frac{(1+\beta_1)\del }{(1-\del p)} + \frac{(1-\beta_1)(2-\del)}{(1-(2-\del p))}$ and $d(p) = \frac{(1+\beta_1)(2-\del)}{(1-(2-\del)p)} + \frac{(1-\beta_1)\del}{(1-\del p)}$.
\\

Then we have 
\begin{align*}
\Pr[d(v)=0| d(v) \le 1, \beta_1] &= \frac{\Pr[d(v)=0 | \beta_1]   }{\Pr[d(v)=0 | \beta_1]+\Pr[d(v)=1 | \beta_1]     }  \\
&= \frac{1+\beta_1^2 n_1^2 (\del-1)^2 p^2/2 + O(\beta_1^4 p^4 n_1^p)}{O(n_1 p) \left[c(p)\left(1+ \frac{\beta_1^2 n_1^2 (\del-1)^2 p^2}{2} + O( \beta_1^4 n_1^4 p^4 )\right) + (d(p) - c(p)) \left( \frac{1-\del p}{1 - (2-\del)p}\right) ^{-\beta_1 n_1/2} \right]}.
\end{align*}
and
\begin{align*}
\Pr[d(v)=1| d(v) \le 1, \beta_1] &= \frac{\Pr[d(v)=1 | \beta_1]   }{\Pr[d(v)=0 | \beta_1]+\Pr[d(v)=1 | \beta_1]     }  \\
&= \frac{\frac{n_1 p}{2} \left[c(p)\left(1+ \frac{\beta_1^2 n_1^2 (\del-1)^2 p^2}{2} + O( \beta_1^4 n_1^4 p^4 )\right) + (d(p) - c(p)) \left( \frac{1-\del p}{1 - (2-\del)p}\right) ^{-\beta_1 n_1/2} \right] }{O(n_1 p) \left[c(p)\left(1+ \frac{\beta_1^2 n_1^2 (\del-1)^2 p^2}{2} + O( \beta_1^4 n_1^4 p^4 )\right) + (d(p) - c(p)) \left( \frac{1-\del p}{1 - (2-\del)p}\right) ^{-\beta_1 n_1/2} \right]}.
\end{align*}

\begin{align*} \text{Now}
\Pr [G | \hat G, \beta_1] &= \Pr[d(v)=0| d(v) \le 1, \beta_1]^{|V_2^{(0)}| } \cdot \Pr[d(v)=1| d(v) \le 1, \beta_1]^{|V_2^{(1)}| }.
\end{align*}
\noindent
Given $\sigma$, $\beta_1$ is determined, and whp over the choice of $\sigma$, $\beta_1^4 \le n_1^{-9/5}$.  Similarly, whp over the choice of $\sigma_U$, the conditional expectation of $\beta_1^4$ is $\le n_1^{-9/5}$.   
All together this gives that whp, 
\begin{align*}
\frac{ \Pr[ G | \sigma, \hat G]  } { \Pr[ G | \sigma_U, \hat G]  } &= (1 + O(\beta_1^4 n_1^4 p^4)) ^{ O(n_2 )  } \\
 &=( 1 + O(n_1^{-9/5} n_1^4 n_1^{-2} n_2^{-2})) ^{ O(n_2 )  } \\
 &= 1+ O( n_1^{1/5} n_2^{-1}) = 1+ o(1)
\end{align*}

This proves Lemma \ref{lem:degree01}. 
\end{proof}


\section{Proof of Lemma \ref{lem:norms}}
\label{sec:proveBBound}

We use another auxiliary lemma, a high probability bound on the norm of a random matrix with mean $0$ independent entries. Such a lemma is proved for Bernoulli random entries in \cite{vu2014simple, crv}, here we extend it to Poisson entries.

\begin{lemma}
\label{lem:binomialNorm}
Let $E$ be an $n \times n$ symmetric random matrix with zeros on the diagonal and independent entries $e_{ij}$ above the diagonal which take the values $X_{ij}- \lam_{ij}$ where each $X_{ij}$ is a Poisson random variable with mean $\lam_{ij}$.   Then there is a constant $C>0$, so that if $\sigma^2 := \max_{ij} \lam_{ij}$, and $\sigma^2 \ge C \log n /n$, 
\[ \Pr[ \| E \| > C \cdot T \cdot \sigma \sqrt n ] \le n^{-T} \]
for any $T \ge 1$. 
\end{lemma}

\begin{proof}
If $\sigma^2 \ge \log ^4 n / \sqrt n $, then we can apply Theorem 1.5 from \cite{vu2005spectral} (the failure probability can be made as small as $n^{-T}$ with the additional factor $T$ in the bound).  

 For $\sigma ^2 \le \log ^4 n / \sqrt n$, we truncate each $X_{ij}$ by writing
\[ X_{ij} = \hat X_{ij} + \tilde X_{ij}  \]
where $\hat X_{ij} = \min  \{ X_{ij}, 1\}$ and $\tilde X_{ij} = X_{ij} - \hat X_{ij}$.  We then define two matrices $\hat E$ and $\tilde E$  with $\hat E_{ij} = \hat X_{ij} - \E \hat X_{ij}$ and $\tilde E_{ij} = \tilde X_{ij} - \E \tilde X_{ij}$.  Thus $E = \hat E + \tilde E$, and so we will bound $\| E \| \le \| \hat E \| + \| \tilde E \|$. 

  Note that $\E \hat X_{ij} = \E X_{ij} (1+o(1))$, and each $\hat X_{ij}$ is a Bernoulli random variable, and so we can apply Lemma 3.4 from \cite{vu2014simple} to get $\| \hat E \| \le  C T \sigma \sqrt {n}$ with probability $\ge 1 -n^{-T}$ (again inspecting the details of the proof in \cite{vu2014simple} gives a failure probability of $n^{-T}$ at the expense of the extra factor $T$ in the bound).  

To bound $\| \tilde E \|$, consider one row sum, $ \left| \sum_{j} \tilde X_{ij} - \E \tilde X_{ij} \right |$. For $j=1 \dots n$, the random variables $ \tilde X_{ij} - \E \tilde X_{ij}$ are independent, mean $0$ random variable with variance $O(\sigma^4)$ with Poisson tails, and so a Chernoff bound gives
\begin{align*}
\Pr \left[  \left| \sum_{j} \tilde X_{ij} - \E \tilde X_{ij} \right | > CT \sigma \sqrt{n}  \right ]  &\le \exp  \left( - \frac{C' T^2}{\sigma^2} \right ) \le \exp \left( -C' T^2 \frac{\sqrt n}{\log^4 n} \right),
\end{align*}
and so with probability at least $1- n^{-T}$ all row sums (and thus $\| \tilde E \|$) are at most $CT \sigma \sqrt{n} $.

\end{proof}

\noindent
With this we prove Lemma \ref{lem:norms}.

\begin{proof}[Proof of Lemma \ref{lem:norms}]

First, note that under the conditions of Theorem~\ref{thm:main}, $n_2 \ge n_1 \log ^4 n_1$, so $n_1 p < 1/
\log n_1$.  (In fact, cases in which the density is much higher than this can be dealt with by the standard method of bounding $\| M - \E M \|$). 

(1): We can compute $(\E B)_{ij}$: this is the expected number of paths of length $2$ from $i$ to $j$ in $G$.  Say $\sigma(i) = \sigma(j)$, $i \ne j$. Then,
\begin{align*}
(\E B)_{ij} &= \frac{n_2}{2} \del^2 p^2 + \frac{n_2}{2} (2-\del)^2 p^2 = n_2 p^2 (\del^2 - 2\del  +2).
\end{align*}
If $\sigma(i) \ne \sigma(j)$, then
\begin{align*}
(\E B)_{ij} &= n_2 p^2 (2\del -\del^2).
\end{align*}

The diagonal entries of $\E B$ are  $0$ by construction. 
So $\E B $ is a rank $2$ matrix, $\E B = \lam_1 J/n_1 + \lam_2 \sigma \sigma^T /n_1$, with $\lam_1 = n_1 n_2 p^2$ and $\lam_2 = (\del-1)^2 n_1 n_2 p^2$.  


(2): The matrix $B - \E B$ is symmetric with mean zero entries, but the entries are not quite independent, and so we cannot directly apply a bound like Lemma \ref{lem:binomialNorm}. 
Instead, we will first decompose the matrix into the sum of a sequence of adjacency matrices of subgraphs induced by vertices of a given degree in $V_2$.  We will couple each matrix in the sum to a matrix with independent entries, apply Lemma \ref{lem:binomialNorm} to each, then take the sum of these bounds as our upper bound on $\| B - \E B \|$.

We  decompose the graph $G$ by sorting the vertices of $V_2$ by degree.   Let $V_2^{(i)}$, $i=1, 2 \dots$ be the set of vertices in $V_2$ of degree $i$. 

Let $M_i$ be the adjacency matrix of $G$ induced by $V_2^{(i)}.$ The main idea of the decomposition is that $M_1$ does not contribute to $B$, as its edges only appear on the diagonal of $M M^T$, and that nearly all of the remaining edges in the graph are in $M_2$.   We write 
\[M = M_1 + M_2 + M_3 +\dots \]
and 
\[ MM^T = M_1 M_1^T + M_2 M_2^T + M_3 M_3^T + \dots \]
The cross terms disappear since the matrices are supported on disjoint sets of columns.

Recall $B = M M^T - \text{diag}(M M^T)$, and $B_i = M_i M_i^T - \text{diag}(M_iM_i^T)$.  We have
\[ B= B_2 + B_3 + \dots \]
since $M_1 M_1 ^T$ is a diagonal matrix.

 A vertex $v \in V_2 ^{(i)}$ has exactly $i$ neighbors, and its contribution to $B_i$ is $1$ in each entry $(u,w)$ where $u \ne w$ are neighbors of $v$. Call $M_v$ the adjacency matrix induced by $v$, and $B_v = M_v M_v^T -\text{diag}M_v M_v^T$.  We have $B_i = \displaystyle\sum_{v \in V_2^{(i)}} B_v$.   $B_v$ has $\binom i 2$ $1$'s above the diagonal.  Now for each $i \ge 3$, we randomly split each $B_v$, $v \in V_2^{(i)}$, into $\binom i 2$ symmetric matrices $B_v^{(1)}, \dots , B_v^{\binom i 2}$ by randomly assigning each of the $1$'s above the diagonal to a unique matrix, along with the symmetric $1$ below the diagonal.  Then we combine the matrices as 
 \begin{align*}
B_i^{(j)}  &= \sum _{v \in V_2^{(i)}} B_v^{(j)}.
\end{align*}

Each $B_i$ is the adjacency matrix of a random graph formed by adding a given number of random $i$-cliques to an empty graph.  Clearly there are correlations between edges in such a graph, and so the purpose of this decomposition is to split the graph into $\binom i 2$ graphs, $G_i^{(j)}$, $j=1 \dots \binom i 2$, with independent edges. This gives a sequence of adjacency matrices $B_i^{(1)}, \dots, B_i^{(\binom i 2 )}$ with identical distributions, but not independent across the different matrices.

All together, we write
\[
B = B_2 + \sum_{i =3}^\infty \sum_{j=1}^{\binom i 2} B_i ^{(j)}
\]

and 

\begin{equation}
\label{eq:Bdecomp}
B = \E B  + (B_2 - \E B_2) + \sum_{i =3}^\infty \sum_{j=1}^{\binom i 2} \left( B_i ^{(j)} - \E B_i ^{(j)} \right)
\end{equation}
where 
\[ \E B = \E B_2 + \sum _{i =3}^ \infty \sum_{j=1}^{\binom {i}{2}} \E B_i^{(j)}. \]

What remains is to bound the norms of the mean zero random matrices in the decomposition above: $\| B_2 - \E B_2 \|$, and the $\|  B_i ^{(j)} - \E B_i ^{(j)} \|$'s.

Let $L_+^{(i)}, L_-^{(i)}$ be the number of vertices of $V_2$ with degree $i$ and label $+$ or $-$ respectively.  As a first step, we calculate the expectations of these random variables:

\begin{align}
\label{eq:Lpbound}
\E | L_+^{(i) }| = \E | L_-^{(i) }| &= \frac{n_2}{2} \Pr [ d(v \in V_2) = i ] \\
\nonumber
 &=\frac{n_2}{2}  \Pr[ \text{Bin}(n_1/2, \del p) + \text{Bin}(n_1/2, (2-\del)p) = i ] \\
 \nonumber
 &= \frac{n_2}{2} \Pr[ \text{Poisson}(n_1 p) =i] (1+ o(1)) \\
 \nonumber
 &=\frac{n_2}{2}  \frac{ e^{-n_1 p} (n_1 p)^i  }{i! } (1+ o(1)) \\
 \nonumber
 &=  \frac{ n_2 (n_1 p)^i  }{2 \cdot i! } (1+ o(1))
\end{align}
where we use the total variation distance bound on a Poisson approximation of a binomial, along with our assumption $n_1 p =o( 1)$. Let $i_0$ be the smallest $i$ so that $\frac{\E |L_+^{(i) }|}{n_1} \le \log n_1$.  
Considering the sum $B_{\ge i_0} =\displaystyle\sum _{i =i_0}^ \infty \sum_{j=1}^{\binom i 2}  B_i ^{(j)}$, we see that the expected row sums of $B_{\ge i_0}$ are bounded by $O(\log n_1)$, and so from a Chernoff bound, with probability at least $1- n_1^{-2}$, all row sums are $O( \log n_1)$.  This gives 
\begin{equation}
\label{eq:boundhighi}
 \sum _{i =i_0}^ \infty \sum_{j=1}^{\binom i 2} \|  B_i ^{(j)} - \E B_i ^{(j)} \| \le \sum _{i =i_0}^ \infty \sum_{j=1}^{\binom i 2} \|  B_i ^{(j)} \|  + \| \E B_i ^{(j)} \|
 \le C \log n_1 \le C n_1^{1/2} n_2^{1/2} p.
 \end{equation}
 with probability $\ge 1- n_1^{-2}$.

Now consider $i < i_0$. Let $N_i^{(j)}(++), N_i^{(j)}(--), N_i^{(j)}(+-)$ be the number of edges between vertices with the respective labels in the graph $G_i^{(j)}$ corresponding to $B_i^{(j)}$. Conditioned on $N_i^{(j)}(++), N_i^{(j)}(--), N_i^{(j)}(+-)$ the edges are distributed uniformly (with replacement) in the respective categories.  Alternatively, consider the adjacency matrices $\tilde B_2, \tilde B_i^{(j)}$, where each edge $(u,v)$ appears with multiplicity according to an independent Poisson random variable of mean $(\E B_i^{(j)})_{uv}$ (one of two values depending on whether $u$ and $v$ have the same or opposite label). Again in this setting, conditioned on the number of edges of each type, $\tilde N_i^{(j)}(++),  \tilde N_i^{(j)}(--), \tilde N_i^{(j)}(+-)$, the edges are distributed uniformly with replacement in the respective categories.

Since  $\E \tilde B_2 = \E B_2 $ and $\E \tilde B_i^{(j)}=\E  B_i^{(j)}$ by construction, we can apply Lemma \ref{lem:binomialNorm} and choose $C$ large enough to get
 \begin{equation}
 \label{b2ineq}
  \| \tilde B_2 - \E B_2 \| \le C n_1^{1/2} \sqrt{4 \E | L_+^{(2)}   |/n_1^2  } 
 \end{equation}
with probability at least $1- n_1^{-8}$, as each entry has variance bounded by $4 \E | V_2^{(2)}   |/n_1^2$, 
and similarly
\begin{equation}
\label{biineq}
 \| \tilde B_i^{(j)} - \E B_i^{(j)}\| \le C \cdot i \cdot n_1^{1/2} \sqrt{ \frac{ 4 \E | L_+^{(i) }|  }{\binom {i}{2} \cdot n_1^2 } } 
 \end{equation}
with probability at least $1 - n_1^{-6 -i}$.  Note that from (\ref{eq:Lpbound}), the means $\E | L_+^{(i)} |$ decrease with $i$ faster than $1/ i!$ and so summing $(\ref{biineq})$ over $3 \le i \le i_0$, and all $j$, gives a bound of $C n_1^{1/2} \sqrt{ \E | L_+^{(3)}|} \le C n_1^{1/2} n_2^{1/2} p$.

To transfer these bounds, we couple the Poisson matrices with  $B_2$ and the $B_i^{(j)}$'s.  If the means of $N_i^{(j)}(++), N_i^{(j)}(--), N_i^{(j)}(+-)$ are small enough, we can couple the matrices to be equal whp.  If the means are large, we couple so that $\| \tilde B_i^{(j)} - B_i^{(j)}\|$ is small.  Take $N_i^{(j)}(++)$.  Its distribution is a $\text{Bin}(n_2,q)$ for $q$ that depends on $n_1, p, i,$ and $\del$.  The corresponding random variable, $\tilde N_i^{(j)}(++)$ is a $\pois(n_2 q)$. Say $q = o(n_2^{1/2})$.  In this case the total variation distance between the two is $O(n_2 q^2)$ and so we can couple the corresponding matrices to be equal whp. $q$ is decreasing like $1/i!$, and so we can sum the deviation probabilities over all $i$ and $j$.  When $q = \Omega(n_2^{-1/2})$, we 
write $N_i^{(j)}(++)$ as the sum of $n_2$ independent $\text{Ber}(q)$
random variables and $\tilde N_i^{(j)}(++)$ as the sum of $n_2$ independent $\pois(q)$ random variables, and term by term in each sum couple by an optimal coupling with respect to total variation distance.  Then the difference $ N_i^{(j)}(++) - \tilde N_i^{(j)}(++)$ is the sum of $n_2$ mean $0$ random variables of variance $O(q^2)$, and so whp the difference is bounded by $O(q n_2^{2/3})$. We can couple the matrices so  that their difference has non-zero entries distributed uniformly in entries corresponding to $+, +$ labels.  Then, as above, a Chernoff bound shows that the row sums (and thus the matrix norm) of the difference matrix are all bounded by $O(q n_1^{-1} n_2^{2/3} \log n_1)$.  Since $q < n_1^2 p^2$, this gives a bound on the norm of $O( n_1 n_2^{2/3} p^2) = o(n_1^{1/2} n_2^{1/2} p  )$.

All together the bound (\ref{eq:boundhighi}) and the transferred bounds give

\begin{equation*}
\| B -\E B \| \le C n_1^{1/2} n_2^{1/2} p
\end{equation*}
with probability at least $1- n_1^{-1}$, which completes the proof of part 2 of Lemma \ref{lem:norms}.
\\

Parts 3 and 4 of Lemma \ref{lem:norms} follow from the observation that the $i$th diagonal entries of $D_V$ are the degrees of the $i$th vertex of $V_1$.  Since the degrees of vertices of $V_1$ have identical distributions, the expectation matrices are multiples of the identity.  
For part 4 we use Chernoff bounds. We have $ \| D_V - \E D_V \| = \displaystyle\text{max}_i |(D_V)_{ii} - n_2p| \le C\sqrt{n_2 p \log n_1}$ whp.
\end{proof}

\end{document}